\documentclass[a4paper]{article}

\usepackage{amsfonts}
\usepackage{amsmath}
\usepackage{amssymb}
\usepackage{amsthm}
\usepackage[UKenglish]{babel}
\usepackage{Baskervaldx}
\usepackage[T1]{fontenc}
\usepackage{tikz}
\usepackage{titling}
\usepackage{tocloft}
\usepackage{xcolor}
	\definecolor{Blue}{HTML}{3d25b9}
	\definecolor{Red}{HTML}{c00054}

\usepackage{enumitem}
\usepackage[hang,flushmargin]{footmisc}
\usepackage[margin=25mm]{geometry}
\usepackage{microtype}
\usepackage{sectsty}
	\sectionfont{\large}
	\subsectionfont{\normalsize}
	\subsubsectionfont{\normalsize}
\usepackage{titlesec}
	\titlespacing{\section}{0pt}{12pt}{0pt}
	\titlespacing{\subsection}{0pt}{6pt}{0pt}

\setlist{topsep=0pt,itemsep=0pt}

\DeclareMathOperator\Jac{Jac}
\newcommand{\Ch}{\mathrm{Chiodo}}

\newcommand{\oM}{\overline{\mathcal M}}
\newcommand{\Z}{\mathbb{Z}}
\newcommand{\ch}{\mathrm{ch}}
\newcommand{\Pic}{\operatorname{Pic}}
\newcommand{\oPic}{\overline{\Pic}}
\newcommand{\Klog}{\omega_{\rm log}}
\newcommand{\cM}{{\mathcal{M}}}
\newcommand{\cL}{{\mathcal{L}}}
\newcommand{\SSS}{\mathcal{S}}

\usepackage[colorlinks=true,linkcolor=Blue,citecolor=Blue]{hyperref}
	\hypersetup{breaklinks=true}
\usepackage[capitalise,noabbrev]{cleveref}
	\crefname{equation}{equation}{equations}

\theoremstyle{plain}
	\newtheorem{theorem}{Theorem}
	\newtheorem{proposition}[theorem]{Proposition}
	
	\newtheorem{lemma}[theorem]{Lemma}
	\newtheorem{conjecture}[theorem]{Conjecture}
	\numberwithin{theorem}{section}
\theoremstyle{definition}
	\newtheorem{definition}[theorem]{Definition}
	
\theoremstyle{remark}
	\newtheorem*{remark*}{Remark}
	\newtheorem{remark}[theorem]{Remark}

\title{On the Goulden--Jackson--Vakil conjecture for double Hurwitz numbers}
\author{Norman Do and Danilo Lewa\'{n}ski}

\begin{document}

\textbf{{\large \thetitle}}

\textbf{\theauthor}

School of Mathematics, Monash University, VIC 3800 Australia \\
Email: \href{mailto:norm.do@monash.edu}{norm.do@monash.edu}

Universit\'e Paris-Saclay, CNRS, CEA, Institut de physique th\'eorique (IPhT), 91191 Gif-sur-Yvette, France. \\
Institut des Hautes \'Etudes Scientifiques, le Bois-Marie, 35 route de Chartres, 91440 Bures-sur-Yvette, France. \\
Email: \href{mailto:danilo.lewanski@ipht.fr}{danilo.lewanski@ipht.fr}

{\em Abstract.} Goulden, Jackson and Vakil observed a polynomial structure underlying one-part double Hurwitz numbers, which enumerate branched covers of $\mathbb{CP}^1$ with prescribed ramification profile over $\infty$, a unique preimage over 0, and simple branching elsewhere. This led them to conjecture the existence of moduli spaces and tautological classes whose intersection theory produces an analogue of the celebrated ELSV formula for single Hurwitz numbers.

In this paper, we present three formulas that express one-part double Hurwitz numbers as intersection numbers on certain moduli spaces. The first involves Hodge classes on moduli spaces of stable maps to classifying spaces; the second involves Chiodo classes on moduli spaces of spin curves; and the third involves tautological classes on moduli spaces of stable curves. We proceed to discuss the merits of these formulas against a list of desired properties enunciated by Goulden, Jackson and Vakil. Our formulas lead to non-trivial relations between tautological intersection numbers on moduli spaces of stable curves and hints at further structure underlying Chiodo classes. The paper concludes with generalisations of our results to the context of spin Hurwitz numbers.

{\em Acknowledgements.} The first author was supported by the Australian Research Council grant DP180103891. The second author was supported by: the Max-Planck-Gesellschaft; the ERC Synergy grant ``ReNewQuantum'' at IPhT and IH\'{E}S, Paris, France; and the Robert Bartnik Visiting Fellowship at Monash University, Melbourne, Australia. The authors are grateful to Gaëtan Borot, Alessandro Chiodo, Alessandro Giacchetto, Paul Norbury, Sergey Shadrin, Mehdi Tavakol, and Dimitri Zvonkine for valuable discussions. We are moreover grateful to Johannes Schmitt for the great help and effort adapting the Sage package \textit{admcycles} to include the possibility of intersecting Chiodo classes of arbitrary parameters, which allowed to check numerically cases of Theorem \ref{thm:comparison} and we believe will turn out valuable many times in the future.

{\em 2010 Mathematics Subject Classification.} 14H10, 14N10, 05A15

\vspace{10pt}

\hrule

\vspace{8pt}

\tableofcontents

\section{Introduction} \label{sec:introduction}

The enumeration of branched covers of Riemann surfaces dates back to Hurwitz~\cite{hur} but has seen a revival in recent decades due to connections with moduli spaces of curves~\cite{ELSV}, integrability~\cite{oko}, and mathematical physics~\cite{bou-mar,eyn-mul-saf}. One catalyst for this renaissance was the discovery of the ELSV formula, which expresses single Hurwitz numbers as tautological intersection numbers on moduli spaces of curves.

The single Hurwitz number $h_{g; \mu}$ is the weighted enumeration of connected genus $g$ branched covers of $\mathbb{CP}^1$ with ramification profile $\mu$ over $\infty$, such that all other ramification is simple and occurs over prescribed points of $\mathbb{C}^*$. We attach the following weight to such a branched cover $f$.
\begin{equation} \label{eq:hurwitzweight}
\frac{|\mathrm{Aut}(\mu)|}{|\mathrm{Aut}(f)| \times (2g - 2 + \ell_0 + \ell_\infty)!}
\end{equation}
Here, $\ell_0$ and $\ell_\infty$ denote the numbers of preimages of 0 and $\infty$, respectively. The expression $2g - 2 + \ell_0 + \ell_\infty$ appearing in the weight is thus equal to the number of simple branch points, as specified by the Riemann--Hurwitz formula. The group $\mathrm{Aut}(\mu)$ comprises permutations of $\mu = (\mu_1, \ldots, \mu_n)$ that leave the tuple fixed, while the group $\mathrm{Aut}(f)$ attached to the branched cover $f: C \to \mathbb{CP}^1$ comprises Riemann surface automorphisms $\phi: C \to C$ that satisfy $f \circ \phi = f$. The factor $\frac{1}{|\mathrm{Aut}(f)|}$ appearing in~\cref{eq:hurwitzweight} is natural from the perspective of enumerative geometry, while the other factors produce a normalisation that makes the structure of single Hurwitz numbers more transparent. For a more thorough introduction to Hurwitz numbers, we point the reader to the literature~\cite{lan-zvo,cav-mil}.

The celebrated ELSV formula expresses single Hurwitz numbers as tautological intersection numbers on moduli spaces of stable curves in the following way.

\begin{theorem}[Ekedahl, Lando, Shapiro and Vainshtein~\cite{ELSV}]
For integers $g \geq 0$ and $n \geq 1$ with $(g,n) \neq (0,1)$ or $(0,2)$, the single Hurwitz numbers satisfy
\begin{equation} \label{eq:ELSV}
h_{g; \mu_1, \ldots, \mu_n} = \prod_{i=1}^n \frac{\mu_i^{\mu_i}}{\mu_i!} \int_{\oM_{g,n}} \frac{\sum_{k=0}^g (-1)^k \lambda_{k}}{\prod_{i=1}^n (1 - \mu_i \psi_i)}.
\end{equation}
\end{theorem}

The integral in~\cref{eq:ELSV} is over the moduli space of stable curves $\overline{\mathcal M}_{g,n}$ and yields a symmetric polynomial in $\mu_1, \ldots, \mu_n$ of degree $3g-3+n$. This polynomiality had previously been observed in small cases and conjectured in general by Goulden, Jackson and Vainshtein~\cite{gou-jac-vai}. More generally, other ``ELSV formulas'' exist, which relate enumerative problems to intersection theory on moduli spaces, such as the formula for orbifold Hurwitz numbers of Johnson, Pandharipande and Tseng~\cite{joh-pan-tse}.

It is natural to consider double Hurwitz numbers $h_{g; \mu, \nu}$, which enumerate connected genus $g$ branched covers of $\mathbb{CP}^1$ with ramification profiles $\mu$ and $\nu$ over $\infty$ and 0, respectively, such that all other ramification is simple and occurs over prescribed points of $\mathbb{C}^*$. The weight attached to such a branched cover is taken to be precisely as in~\cref{eq:hurwitzweight}. Although double Hurwitz numbers have received some attention in the literature, various open questions remain. In particular, the present work focuses on the compelling Goulden--Jackson--Vakil conjecture concerning one-part double Hurwitz numbers, which are defined as follows.

\begin{definition}
Let $h_{g; \mu}^{\textnormal{one-part}}$ denote the double Hurwitz number $h_{g; \mu; \nu}$, where $\nu$ is the partition with precisely one part, which is equal to $|\mu|$. (For a partition $\mu$, we use the standard notation $|\mu| = \mu_1 + \mu_2 + \cdots + \mu_{\ell(\mu)}$, where $\ell(\mu)$ denotes the number of parts of $\mu$.)
\end{definition}

Goulden, Jackson and Vakil proved that for fixed $g$ and $n$, the one-part double Hurwitz number $h_{g; \mu_1, \ldots, \mu_n}^{\textnormal{one-part}}$ is a polynomial in $\mu_1, \ldots, \mu_n$. More precisely, it is of the form $\mu_1 + \cdots + \mu_n$ multiplied by a polynomial of degree $4g-3+n$ in which all monomials have the same parity of degree. Thus, they were led to conjecture the following, in direct analogy with the original ELSV formula for single Hurwitz numbers.

\begin{conjecture}[Goulden, Jackson and Vakil~\cite{gou-jac-vak}] \label{con:GJV}
For integers $g \geq 0$ and $n \geq 1$ with $(g,n) \neq (0,1)$ or $(0,2)$, there exists a moduli space $\oPic_{g,n}$ with classes $\Lambda_{2k} \in H^{4k}(\oPic_{g,n})$ and $\Psi_i \in H^2(\oPic_{g,n})$ such that the one-part double Hurwitz numbers satisfy
\[
h_{g; \mu_1, \ldots, \mu_n}^{\textnormal{one-part}} = (\mu_1 + \cdots + \mu_n) \int_{\oPic_{g,n}} \frac{\sum_{k=0}^g (-1)^k \Lambda_{2k}}{\prod_{i=1}^n (1 - \mu_i \Psi_i)}.
\]
\end{conjecture}

The space $\oPic_{g,n}$ and its classes $\Lambda_{2k}$ and $\Psi_i$ are yet to be defined, but expected to satisfy several natural properties, which are listed in~\cref{sec:properties}. For example, the polynomial structure of one-part double Hurwitz numbers suggests that $\oPic_{g,n}$ carries a virtual fundamental class of complex dimension $4g-3+n$. Perhaps the most speculative of the aforementioned properties is the fact that $\oPic_{g,n}$ should be a compactification of the universal Picard variety $\Pic_{g,n}$, the moduli space that parametrises genus $g$ curves with $n$ marked points equipped with a degree 0 line bundle. One would then expect that forgetful morphisms, $\psi$-classes and $\lambda$-classes on moduli spaces of stable curves $\overline{\mathcal M}_{g,n}$ would have natural analogues that exhibit similar behaviour on $\oPic_{g,n}$.

Our main result is~\cref{thm:main}, which comprises three related formulas that serve as candidates for an ELSV formula for one-part double Hurwitz numbers. Respectively, they express $h_{g; \mu}^{\textnormal{one-part}}$ in terms of
\begin{itemize}
\item Hodge classes on moduli spaces of stable maps to classifying spaces;
\item Chiodo classes on moduli spaces of spin curves; and
\item tautological classes on moduli spaces of stable curves.
\end{itemize}
The proof of~\cref{thm:main} is accomplished by specialising known results from the literature, particularly the work of Johnson, Pandharipande and Tseng on abelian Hurwitz--Hodge integrals~\cite{joh-pan-tse}, as well as the work of Popolitov, Shadrin, Zvonkine and the second author on Chiodo classes~\cite{LPSZ}.

The current work provides a partial resolution to~\cref{con:GJV}. However, let us state from the outset that none of our formulas satisfies all of the desired properties enunciated by Goulden, Jackson and Vakil. In~\cref{sec:properties}, we discuss the relative merits of these formulas against these properties. For example, one of our formulas expresses $h_{g; \mu_1, \ldots, \mu_n}^{\textnormal{one-part}}$ as an integral over $\overline{\mathcal M}_{g,n+g}$. Thus, the moduli space possesses the virtuous features of having the expected dimension $4g-3+n$ and not depending on the partition $\mu$. On the other hand, it is not explicitly a moduli space of line bundles, as the universal Picard variety should be.

Despite the previous remarks, it is conceivable that~\cref{thm:main} may lead to a more satisfying resolution of the Goulden--Jackson--Vakil conjecture. Furthermore, our results have direct applications to intersection theory on moduli spaces of curves. By invoking the obvious symmetry for double Hurwitz numbers $h_{g; \mu; \nu} = h_{g; \nu; \mu}$, which exchanges ramification profiles over $0$ and $\infty$, we are able to compare instances of the original ELSV formula for single Hurwitz numbers with instances of our formula for one-part double Hurwitz numbers. This leads to a non-trivial relation between intersection numbers on $\overline{\mathcal M}_{g,1}$ and $\overline{\mathcal M}_{g,d}$, stated as~\cref{thm:comparison}. It is natural to wonder whether such a result may be the shadow of a richer relation at the level of the cohomology of $\overline{\mathcal M}_{g,n}$.

It is known from the work of Goulden, Jackson and Vakil that one can in fact write
\[
h_{g; \mu_1, \ldots, \mu_n}^{\textnormal{one-part}} = (\mu_1 + \cdots + \mu_n)^{2g-2+n} \, P_{g,n}(\mu_1^2, \ldots, \mu_n^2),
\]
for some symmetric polynomial $P_{g,n}$ of degree $g$~\cite{gou-jac-vak}. We make the observation that such structure is not self-evident from any of our ELSV formulas for one-part double Hurwitz numbers. So this polynomiality then suggests some further structure underlying the Chiodo classes that arise in our ELSV formulas and we leave the implications to future work.

One may consider analogues of the Hurwitz numbers above with the simple branch points replaced by branch points with ramification profile $(r+1, 1, 1, \ldots, 1)$ for some fixed positive integer $r$. The work of Okounkov and Pandharipande on the Gromov--Witten/Hurwitz correspondence suggests that it is geometrically natural to compactify such a count using so-called completed cycles and the resulting enumeration yields spin Hurwitz numbers~\cite{oko-pan}. Zvonkine conjectured an ELSV formula in the spin setting~\cite{zvo}, and this conjecture was later generalised further to spin orbifold Hurwitz numbers~\cite{KLPS}. These conjectures were ultimately resolved in a series of five papers involving the work of Borot, Dunin-Barkowski, Kramer, Popolitov, Shadrin, Spitz, Zvonkine and the second author~\cite{SSZ2, LPSZ, KLPS, BKLPS, DKPS}. This allows us to produce spin orbifold analogues of our results, with which we conclude the paper.

The structure of the paper is as follows.
\begin{itemize}
\item In~\cref{sec:background}, we briefly introduce the moduli spaces and associated cohomology classes that appear in our main result. The exposition is necessarily concise, serving only to recall the relevant definitions and notations. We include references to the literature for the reader seeking a more thorough treatment.

\item In~\cref{sec:elsv}, we state and prove the main result of the paper ---~\cref{thm:main} --- which comprises three candidates for an ELSV formula for one-part double Hurwitz numbers.

\item In~\cref{sec:properties}, we discuss the relative merits of our formulas against the list of properties sought by Goulden, Jackson and Vakil from an ELSV formula for one-part double Hurwitz numbers. We conclude the section by observing that the polynomiality of one-part double Hurwitz numbers suggests some further structure underlying the Chiodo classes.

\item In~\cref{sec:comparison}, we use the symmetry that exchanges ramification profiles over $0$ and $\infty$ to compare our results with the original ELSV formula. This produces new relations between tautological intersection numbers on moduli spaces of curves. The aforementioned argument is then generalised to the setting of orbifold Hurwitz numbers, by comparing with the Johnson--Pandhandripande--Tseng formula.

\item In~\cref{sec:verification}, we perform some initial calculations that verify the main relation of the previous section in some low genus and low degree cases.

\item In~\cref{sec:spin}, we present a generalisation of our main result to the spin setting, which in turn leads to new relations between tautological intersection numbers on moduli spaces of stable curves. Since the ideas involved are essentially those contained in previous sections, we keep the exposition brief and focus on presenting the relevant results without proof.
\end{itemize}

\section{Background} \label{sec:background}

In this section, we briefly introduce the algebro-geometric objects and corresponding notations required to state our main result. As usual, one can choose to work in terms of cohomology classes or their corresponding Chow classes instead. We have opted for the former and all cohomology is taken with rational coefficients.

\subsection{Tautological classes on moduli spaces of stable curves} \label{subsec:tautological}

Let ${\mathcal M}_{g,n}$ be the moduli space of non-singular algebraic curves $(C; p_1, \ldots, p_n)$ of genus $g$, with $n$ distinct marked points $p_1, \ldots, p_n \in C$. The Deligne--Mumford compactification $\overline{\mathcal M}_{g,n}$ is the moduli space of stable algebraic curves $(C; p_1, \ldots, p_n)$ of genus $g$, with $n$ distinct non-singular marked points $p_1, \ldots, p_n \in C$. A marked algebraic curve is stable if all of its singularities are nodes and there are finitely many automorphisms that preserve the marked points.

The forgetful morphism $\pi: \overline{\mathcal M}_{g,n+1} \to \overline{\mathcal M}_{g,n}$ forgets the point marked $n+1$ and stabilises the curve, if necessary. There is a natural identification of the universal curve $\overline{\mathcal C}_{g,n}$ with $\overline{\mathcal M}_{g,n+1}$, which allows us to define sections $\sigma_1, \ldots, \sigma_n: \overline{\mathcal M}_{g,n} \to \overline{\mathcal M}_{g,n+1}$ corresponding to the marked points of the curve. The relative dualising sheaf ${\mathcal L} = {\mathcal K}_{\overline{\mathcal M}_{g,n+1}} \otimes \pi^* {\mathcal K}_{\overline{\mathcal M}_{g,n}}^{-1}$ extends to the compactification the vertical cotangent bundle on ${\mathcal M}_{g,n+1} \cong {\mathcal C}_{g,n}$, whose fibre over $(C, p)$ is the cotangent line $T^*_p C$.
\begin{itemize}
\item The $\psi$-classes are given by $\psi_i = c_1(\sigma_i^* {\mathcal L}) \in H^2(\overline{\mathcal M}_{g,n})$ for $1 \leq i \leq n$.

\item The $\kappa$-classes are given by $\kappa_m = \pi_*(c_1({\mathcal E})^{m+1}) \in H^{2m}(\overline{\mathcal M}_{g,n})$ for $m \geq 0$, where ${\mathcal E} = {\mathcal L}(\sum \sigma_i(\overline{\mathcal M}_{g,n}))$.

\item The $\lambda$-classes are given by $\lambda_k = c_k(\Lambda) \in H^{2k}(\overline{\mathcal M}_{g,n})$ for $k \geq 0$. Here, $\Lambda = \pi_*({\mathcal L})$ denotes the Hodge bundle, whose fibre over $[C] \in {\mathcal M}_{g,n}$ is the $g$-dimensional vector space of holomorphic 1-forms on $C$.
\end{itemize}

The study of the cohomology ring $H^*(\overline{\mathcal M}_{g,n})$ has received a great deal of attention although an explicit description is widely considered to be untractable at present. One can instead focus on the tautological rings $R^*(\overline{\mathcal M}_{g,n}) \subseteq H^*(\overline{\mathcal M}_{g,n})$, whose classes are geometrically natural in some sense. They are simultaneously defined for all $g$ and $n$ as the smallest system of $\mathbb{Q}$-algebras closed under pushforwards by the natural forgetful and gluing morphisms between moduli spaces of stable curves. All of the classes defined above live in the tautological ring $R^*(\overline{\mathcal M}_{g,n})$. For a more thorough introduction to moduli spaces of stable curves and their tautological rings, the reader is encouraged to consult the literature~\cite{har-mor}.

\subsection{Chiodo classes on moduli spaces of spin curves} \label{subsec:chiodo}

For $2g - 2 +n > 0$, consider a non-singular marked curve $(C; p_1, \ldots, p_n) \in \cM_{g,n}$ and let $\Klog = \omega_C(\sum p_i)$ be its log canonical bundle. Fix a positive integer $r$, and let $1 \leq s \leq r$ and $1 \leq a_1, \ldots, a_n \leq r$ be integers satisfying the equation
\[
a_1 + a_2 + \cdots + a_n \equiv (2g-2+n)s \pmod{r}.
\]
This condition guarantees the existence of a line bundle over $C$ whose $r$th tensor power is isomorphic to $\Klog^{\otimes s}(-\sum a_i p_i)$. Varying the underlying curve and the choice of such an $r$th tensor root yields a moduli space with a natural compactification $\oM_{g;a_1, \ldots, a_n}^{r,s}$ that was independently constructed by Chiodo~\cite{chi1} and Jarvis~\cite{jar}. These works also include constructions of the universal curve $\pi : \overline{\mathcal C}_{g;a_1, \ldots, a_n}^{r,s} \to \oM_{g; a_1, \ldots, a_n}^{r,s}$ and the universal $r$th root $\cL \to \overline{\mathcal C}_{g; a_1, \ldots, a_n}^{r,s}$.

One can define psi-classes and kappa-classes in complete analogy with the case of moduli spaces of stable curves, as described previously. Chiodo's formula then states that the Chern characters of the derived pushforward $\mathrm{ch}_k(R^*\pi_*{\mathcal L})$ are given by
\begin{align} \label{eq:chiodoformula}
\mathrm{ch}_k(r, s; a_1, \ldots, a_n) :={}& \frac{B_{k+1}(s/r)}{(k+1)!} \kappa_k - \sum_{i=1}^n \frac{B_{k+1}(a_i/r)}{(k+1)!} \psi_i^k \notag \\
&+ \frac{r}{2} \sum_{a=0}^{r-1} \frac{B_{k+1}(a/r)}{(k+1)!} {j_a}_* \frac{(\psi')^k + (-1)^{k-1}(\psi'')^k}{\psi'+\psi''}. 
\end{align}
Here, $B_m(x)$ denotes the Bernoulli polynomial, $j_a$ is the boundary morphism that represents the boundary divisor with multiplicity index $a$ at one of the two branches of the corresponding node, and $\psi',\psi''$ are the $\psi$-classes at the two branches of the node~\cite{chi2}.

We will commonly use the class in $H^*(\oM_{g; a_1, \ldots, a_n}^{r,s})$ defined by
\begin{align} \label{eq:chiodo}
\Ch_{g,n}(r, s; a_1, \ldots, a_n) :={}& c(-R^*\pi_*\cL) \notag \\
={}& \exp \bigg[ \sum_{k=1}^\infty (-1)^k (k-1)! \, \ch_k(r, s; a_1, \ldots, a_n) \bigg].
\end{align}
More generally, we also use the notation
\[
\Ch_{g,n}^{[x]}(r, s; a_1, \ldots, a_n) := \exp \bigg[ \sum_{k=1}^\infty (-x)^k (k-1)! \, \ch_k(r, s; a_1, \ldots, a_n) \bigg].
\]
It is natural and convenient to consider $a_1, \ldots, a_n$ modulo $r$ and we will do so throughout. This allows us, for example, to write statements such as $\Ch_{g,n}(r, s; -1, \ldots, -1) \in H^*(\oM_{g; -1, \ldots, -1}^{r,s})$.

There is a natural forgetful morphism
\[
\epsilon: \overline{\mathcal{M}}^{r,s}_{g; a_1, \ldots, a_n} \to \overline{\mathcal{M}}_{g,n},
\]
which forgets the line bundle, otherwise known as the spin structure. It is an $r^{2g}$-sheeted cover, unramified away from the boundary; however, $\epsilon$ in fact has degree $r^{2g-1}$ due to the $\Z_r$ symmetry of each $r$th root generated by a morphism multiplying by a primitive root of unity in the fibres. This forgetful morphism allows us to consider the pushforward of the Chiodo classes to moduli spaces of stable curves.

\subsection{Hodge classes on moduli spaces of stable maps to classifying spaces} \label{subsec:hodge}

For $G$ a finite group, let $\overline{\mathcal M}_{g; \gamma}({\mathcal B}G)$ be the moduli stack of stable maps from a genus $g$ marked curve $(C; p_1, \ldots, p_n)$ to the classifying space ${\mathcal B}G$, with monodromy data $\gamma = (\gamma_1, \ldots, \gamma_n)$, where $\gamma_i$ is the monodromy around the marked point $p_i$.

There is a natural map $\epsilon: \overline{\mathcal M}_{g; \gamma}({\mathcal B}G) \to \overline{\mathcal M}_{g,n}$ that sends a stable map to the stabilisation of its domain curve. One can thus define psi-classes via the pullback construction
\[
\overline{\psi}_i = \epsilon^*(\psi_i) \in H^2(\overline{\mathcal M}_{g; \gamma}({\mathcal B}G)), \qquad \text{for } 1 \leq i \leq n.
\]

In the following, we are only interested in the case $G = \mathbb{Z}_r$ for some positive integer $r$, in which case the monodromy data is given by a tuple $(a_1, \ldots, a_n)$ of integers that we consider modulo $r$. The Hodge bundle $\overline{\Lambda} \to \overline{\mathcal M}_{g; \gamma}({\mathcal B}\mathbb{Z}_r)$ associates to the map $f: [D / \mathbb{Z}_r] \to {\mathcal B}\mathbb{Z}_r$ the $\rho$-summand of the $\mathbb{Z}_r$-representation $H^0(D, \omega_D)$, where $\rho: \mathbb{Z}_r \to \mathbb{C}^*$ is the representation defined by $1 \mapsto \exp(\frac{2\pi i}{r})$. We then define the Hodge classes as
\[
\overline{\lambda}_k = c_k(\overline{\Lambda}) \in H^{2k}(\overline{\mathcal M}_{g; \gamma}({\mathcal B}\mathbb{Z}_r)), \qquad \text{for } k \geq 0.
\]

\section{ELSV formulas for one-part double Hurwitz numbers} \label{sec:elsv}

We are now in a position to state and prove our main result. The proof relies heavily on two results from the literature: the first is an ELSV formula for orbifold Hurwitz numbers using Hodge integrals on moduli spaces of stable maps, proved by Johnson, Pandharipande and Tseng~\cite{joh-pan-tse}; the second is an alternative ELSV formula for orbifold Hurwitz numbers using Chiodo classes, proved by Popolitov, Shadrin, Zvonkine and the second author~\cite{LPSZ}.

\begin{theorem}[ELSV formulas for one-part double Hurwitz numbers] \label{thm:main}
For integers $g \geq 0$ and $n \geq 1$ with $(g,n) \neq (0,1)$ or $(0,2)$, the one-part double Hurwitz numbers satisfy the following formulas, where $d = \mu_1 + \cdots + \mu_n$.
\begin{itemize}
\item Hodge classes on moduli spaces of stable maps to the classifying space $\mathcal{B} \mathbb{Z}_d$
\begin{equation} \label{eq:ELSVdJPT}
h_{g; \mu_1, \ldots, \mu_n}^{\textnormal{one-part}} = d^{2-g} \int_{\overline{\mathcal{M}}_{g; -\mu_1, \ldots, -\mu_n}(\mathcal{B}\mathbb{Z}_d)} \frac{\sum_{k=0}^\infty(-d)^k \overline{\lambda}_k}{\prod_{i=1}^n (1 - \mu_i \bar{\psi}_i)}
\end{equation}

\item Chiodo classes on moduli spaces of spin curves
\begin{equation} \label{eq:ELSVdChiodo}
h_{g; \mu_1, \ldots, \mu_n}^{\textnormal{one-part}} = d^{2-g} \int_{\overline{\mathcal{M}}^{d,d}_{g,n; -\mu_1, \ldots, -\mu_n}} 
\frac{\Ch_{g,n}^{[d]}(d, d; -\mu_1, \ldots, -\mu_n)}{\prod_{i=1}^n (1 - \mu_i \psi_i)}
\end{equation}

\item Tautological classes on moduli spaces of stable curves
\begin{align}
h_{g; \mu_1, \ldots, \mu_n}^{\textnormal{one-part}} &= d^{2-g} \int_{\overline{\mathcal{M}}_{g,n}} 
\frac{\epsilon_* \Ch_{g,n}^{[d]}(d, d; -\mu_1, \ldots, -\mu_n)}{\prod_{i=1}^n (1 - \mu_i \psi_i)} \label{eq:ELSVdMgn} \\
&= d^{2-g} \int_{\overline{\mathcal{M}}_{g,n+g}} \frac{\epsilon_* \Ch_{g,n + g}^{[d]}(d, d; -\mu_1, \ldots, -\mu_n, 0, \ldots, 0)}{\prod_{i=1}^n (1 - \mu_i \psi_i)} \, \mathfrak{c}_{g,n} \label{eq:ELSVdMgng}
\end{align}
\end{itemize}
This 
 last expression uses the class $\mathfrak{c}_{g,n} = \frac{(2g - 2 + n)!}{(3g - 3 + n)!} \, \psi_{n+1} \cdots \psi_{n+g} \in H^{2g}(\overline{\mathcal M}_{g,n+g})$. 
\end{theorem}

\begin{proof}
We begin with the notion of orbifold Hurwitz numbers $h_{g;\mu}^{q\textnormal{-orbifold}}$, which enumerate connected genus $g$ branched covers of $\mathbb{CP}^1$ with ramification profile $\mu$ over $\infty$ and ramification profile $(q, q, \ldots, q)$ over 0, such that all other ramification is simple and occurs over prescribed points of $\mathbb{C}^*$. We take the weight attached to such a branched cover to be precisely as in~\cref{eq:hurwitzweight}.

Johnson, Pandharipande and Tseng prove the following ELSV formula for orbifold Hurwitz numbers, expressing them as intersection numbers on moduli spaces of stable maps to classifying spaces~\cite{joh-pan-tse}.
\begin{equation} \label{eq:JPT}
h_{g;\mu_1, \ldots, \mu_n}^{q\textnormal{-orbifold}} = q^{1-g+|\mu|/q} \prod_{i=1}^n \frac{(\mu_i/q)^{\lfloor \mu_i / q \rfloor}}{\lfloor \mu_i / q \rfloor !} \int_{\overline{\mathcal M}_{g; -\mu_1, \ldots, -\mu_n}({\mathcal B}\mathbb{Z}_q)} \frac{\sum_{k=0}^\infty (-q)^k \overline{\lambda}_k}{\prod_{i=1}^n (1-\mu_i \overline{\psi}_i)}
\end{equation}
Now we simply make the observation that for $d = \mu_1 + \cdots + \mu_n$, we have $h_{g; \mu}^{\textnormal{one-part}} = h_{g;\mu}^{d\textnormal{-orbifold}}$. Specialising the above formula to the case $q = d$ and using the fact that $\frac{(\mu_i/q)^{\lfloor \mu_i / q \rfloor}}{\lfloor \mu_i / q \rfloor !} = 1$ for $\mu_i \leq q$ immediately yields~\cref{eq:ELSVdJPT}.

The following ELSV formula for orbifold Hurwitz numbers is implicit in the work of Popolitov, Shadrin, Zvonkine and the second author, by comparing Theorems 4.5 and 5.1 from~\cite{LPSZ} with $r = s = q$.
\begin{equation} \label{eq:JPTalternative}
h_{g;\mu_1, \ldots, \mu_n}^{q\textnormal{-orbifold}} = q^{2g-2+n+|\mu|/q} \prod_{i=1}^n \frac{(\mu_i/q)^{\lfloor \mu_i / q \rfloor}}{\lfloor \mu_i / q \rfloor !} \int_{\overline{\mathcal M}_{g,n}} \frac{\epsilon_* \Ch_{g,n}(q, q; -\mu_1, \ldots, -\mu_n)}{\prod_{i=1}^n (1 - \frac{\mu_i}{q} \psi_i)} 
\end{equation}
Again, we specialise to the case $q = d$, which removes the product of combinatorial factors preceding the integral.
\[
h_{g; \mu_1, \ldots, \mu_n}^{\textnormal{one-part}} = h_{g;\mu_1, \ldots, \mu_n}^{d\textnormal{-orbifold}} = d^{2g-1+n} \int_{\overline{\mathcal M}_{g,n}} \frac{\epsilon_* \Ch_{g,n}(d, d; -\mu_1, \ldots, -\mu_n)}{\prod_{i=1}^n (1 - \frac{\mu_i}{d} \psi_i)}
\]
One obtains~\cref{eq:ELSVdMgn} from this by multiplying each class of cohomological degree $2k$ in the integrand by~$d^k$. We compensate with a global factor of $d^{-(3g-3+n)}$.

As Theorem 4.5 of~\cite{LPSZ} is obtained via the pushforward of Chiodo classes from the moduli space of spin curves, we immediately have~\cref{eq:ELSVdChiodo}.

Finally, to prove~\cref{eq:ELSVdMgng}, we invoke the dilaton equation for orbifold Hurwitz numbers~\cite[Theorem 20]{do-lei-nor}. One can use~\cref{eq:JPTalternative} to express this in the language of Chiodo classes as follows.
\begin{multline}
 \int_{\overline{\mathcal M}_{g,n+1}} \Ch_{g,n+1}^{[d]}(d, d; -\mu_1, \ldots, -\mu_n, 0) \cdot \psi_1^{a_1} \cdots \psi_n^{a_n} \psi_{n+1} \\
 = (2g-2+n) \int_{\overline{\mathcal M}_{g,n}} \Ch_{g,n}^{[d]}(d, d; -\mu_1, \ldots, -\mu_n) \cdot \psi_1^{a_1} \cdots \psi_n^{a_n}
\end{multline}
Applying this dilaton equation $g$ times demonstrates the equivalence of~\cref{eq:ELSVdMgn,eq:ELSVdMgng}, which completes the proof.
\end{proof}

\section{Properties of the formulas} \label{sec:properties}

In its original form, the Goulden--Jackson--Vakil conjecture ---~\cref{con:GJV} --- predicts an ELSV formula with a particular structure, but leaves some room for freedom. Rather than prescribing the exact geometric ingredients, it describes desirable properties that they are expected to satisfy. In this section, we analyse these properties in turn and discuss the extent to which our proposed ELSV formulas satisfy them. We conclude the section with some remarks on how~\cref{thm:main} suggests potential further structure underlying Chiodo classes.

\subsection{Primary properties} \label{subsec:primary}

Goulden, Jackson and Vakil predict an ELSV formula for one-part double Hurwitz numbers of the following form.
\[
h_{g; \mu_1, \ldots, \mu_n}^{\textnormal{one-part}} = (\mu_1 + \cdots + \mu_n) \int_{\oPic_{g,n}} \frac{\sum_{k=0}^g (-1)^k \Lambda_{2k}}{\prod_{i=1}^n (1 - \mu_i \Psi_i)}
\]
They furthermore posit the following four ``primary'' properties, which are essentially taken verbatim from their paper~\cite[Conjecture~3.5]{gou-jac-vak}.

{\bf Property 1.} {\em There is a moduli space $\oPic_{g,n}$, with a (possibly virtual) fundamental class $[ \oPic_{g,n} ]$ of dimension $4g-3+n$, and an open subset isomorphic to the Picard variety $\Pic_{g,n}$ of the universal curve over $\oM_{g,n}$ (where the two fundamental classes agree).}

{\bf Property 2.} {\em There is a forgetful morphism $\pi: \oPic_{g,n+1} \rightarrow \oPic_{g,n}$ (flat, of relative dimension 1), with $n$ sections~$\sigma_i$ giving Cartier divisors $\Delta_{i,n+1}$ ($1 \leq i \leq n$). Both morphisms behave well with respect to the fundamental class: $[\oPic_{g,n+1}] = \pi^*[\oPic_{g,n}]$, and $\Delta_{i,n+1} \cap \oPic_{g,n+1} \cong \oPic_{g,n}$ (with isomorphisms given by $\pi$ and $\sigma_i$), inducing $\Delta_{i,n+1} \cap [\oPic_{g,n+1}] \cong [\oPic_{g,n}]$.}

{\bf Property 3.} {\em There are $n$ line bundles, which over ${\mathcal M}_{g,n}$ correspond to the cotangent spaces of the first $n$ points on the curve (i.e. over ${\mathcal M}_{g,n}$ they are the pullbacks of the ``usual'' $\psi$-classes on ${\mathcal M}_{g,n}$). Denote their first Chern classes $\Psi_1, \ldots, \Psi_n$. They satisfy $\Psi_i = \pi^* \Psi_i + \Delta_{i,n+1}$ ($i \leq n$) on $\oPic_{g,n+1}$ (the latter $\Psi_i$ is on $\oPic_{g,n}$), and $\Psi_i \cdot \Delta_{i,n+1} = 0$.}

{\bf Property 4.} {\em There are Chow (or cohomology) classes $\Lambda_{2k}$ ($k= 0, 1, \ldots, g$) of codimension $2k$ on $\oPic_{g,n}$, which are pulled back from $\oPic_{g,1}$ (if $g>0$) or $\oPic_{0,3}$; $\Lambda_0 = 1$. The $\Lambda$-classes are the Chern classes of a rank $2g$ vector bundle isomorphic to its dual.}

Below, we briefly discuss the relative merits of the ELSV formulas obtained in~\cref{thm:main} against the four properties above.

{\bf Hodge classes on moduli spaces of stable maps to the classifying space $\mathcal{B} \mathbb{Z}_d$ ---~\cref{eq:ELSVdJPT}.} The moduli space $\overline{\mathcal{M}}_{g; -\mu_1, \ldots, -\mu_n}(\mathcal{B}\mathbb{Z}_d)$ has dimension $3g-3+n$ rather than $4g-3+n$. It can be equivalently described as a moduli space of principal $\mathbb{Z}_d$-bundles over stable curves~\cite{joh-pan-tse}. This makes some thematic connection with the universal Picard variety, which is a moduli space of line bundles over stable curves. Goulden, Jackson and Vakil predict a fixed space $\oPic_{g,n}$ from which all one-part double Hurwitz numbers $h_{g; \mu_1, \ldots, \mu_n}^{\textnormal{one-part}}$ with fixed $g$ and $n$ can be calculated. On the other hand,~\cref{eq:ELSVdJPT} uses a moduli space that depends on $g$ and the tuple $(\mu_1, \ldots, \mu_n)$. As a result, there is no obvious natural forgetful morphism that removes a marked point, even though there are natural psi-classes and lambda-classes. Observe that the classes $\overline{\lambda}_0, \overline{\lambda}_1, \overline{\lambda}_2, \ldots$ play the role of the Hodge classes in~\cref{eq:JPT}, the Johnson--Pandharipande--Tseng formula for orbifold Hurwitz numbers.

{\bf Chiodo classes on moduli spaces of spin curves ---~\cref{eq:ELSVdChiodo}.} The moduli space $\overline{\mathcal{M}}^{d,d}_{g,n; -\mu_1, \ldots, -\mu_n}$ has dimension $3g-3+n$ rather than $4g-3+n$. It shares some commonality with the Picard variety, since it is naturally a moduli space of line bundles over stable curves. Again,~\cref{eq:ELSVdChiodo} uses a moduli space that depends on $g$ and the tuple $(\mu_1, \ldots, \mu_n)$. As a result, there is no obvious natural forgetful morphism that removes a marked point, even though there are natural psi-classes. The Chiodo classes have proved to be important in various contexts, such as in the spin-ELSV formula~\cite{zvo,LPSZ} and the formula for the double ramification cycle~\cite{JPPZ}. In the former case, the Chiodo classes naturally take on an analogous role to the Hodge bundle in the original ELSV formula of~\cref{eq:ELSV}.

{\bf Tautological classes on moduli spaces of stable curves ---~\cref{eq:ELSVdMgn}.} The moduli space $\oM_{g,n}$ has dimension $3g-3+n$ rather than $4g-3+n$ and does not have a natural description as a moduli space of bundles on stable curves. Given that the Goulden--Jackson--Vakil conjecture is modelled on the ELSV formula, which also uses $\oM_{g,n}$, many of the remaining properties are satisfied. Namely, we have a forgetful morphism $\pi: \overline{\mathcal M}_{g,n+1} \to \overline{\mathcal M}_{g,n}$, natural sections $\sigma_i$ with associated Cartier divisors, and psi-classes, with all these geometric constructions behaving well with respect to each other.

Property 4 asks for the analogues of the Hodge classes to arise as Chern classes of vector bundles and to exhibit nice behaviour under pullback. In this case, the Chiodo classes are indeed defined as Chern classes, but of the virtual vector bundle $ - R^0 \pi_* \mathcal{L} + R^1 \pi_* \mathcal{L}$. When one of these two terms vanishes, one obtains an actual vector bundle whose rank can be computed by the Riemann-Roch formula. The Chiodo classes $\epsilon_*\Ch_{g,n}(r,s; a_1, \ldots, a_n)$ do behave well under pullback, since they are known to form a semi-simple cohomological field theory with flat unit for $0 \leq s \leq r$~\cite{LPSZ}.

The moduli space of curves $\overline{\mathcal M}_{g,n}$ seems the natural space for an ELSV formula, especially from the point of view of topological recursion. Indeed, a fundamental theorem of Eynard states that the quantities produced by topological recursion can be expressed as tautological intersection numbers on moduli spaces of curves~\cite{eynard}. One-part Hurwitz numbers may be generated via topological recursion in a somewhat non-standard way. Typically, all numbers of a given enumerative geometric problem are produced from the same spectral curve data used as input to the topological recursion. In this case, however, the numbers $h^{\textnormal{one-part}}_{g; \mu_1, \ldots, \mu_n}$ for fixed $|\mu|$ may be generated by the spectral curve for orbifold Hurwitz numbers given by the following data~\cite{bou-her-liu-mul,do-lei-nor}.
\[
\Sigma = \mathbb{C}\mathbb{P}^1 \qquad \qquad x(z) = z e^{-z^{|\mu|}} \qquad \qquad y(z) = z^{|\mu|} \qquad \qquad B(z_1, z_2) = \frac{\mathrm{d}z_1 \, \mathrm{d}z_2}{(z_1 - z_2)^2}
\]

Therefore, the totality of one-part double Hurwitz numbers is stored in an infinite discrete family of spectral curves, instead of a single one. In particular, two such numbers are produced by the same spectral curve if and only if they depend on partitions of equal size. In recent work with Borot, Karev and Moskovsky, we prove that double Hurwitz numbers in general are governed by the topological recursion, using a family of spectral curves that are coupled with formal variables~\cite{BDKLM}. A consequence of this work is an ELSV-type formula for double Hurwitz numbers and it would be interesting to consider the implications for one-part double Hurwitz numbers, although we do not pursue that line of thought here.

{\bf Tautological classes on moduli spaces of stable curves ---~\cref{eq:ELSVdMgn}.} The moduli space $\oM_{g,n+g}$ has the predicted dimension $4g-3+n$. Although it does not arise naturally as a moduli space of line bundles, we observe the following interplay between the uncompactified moduli space $\mathcal{M}_{g,n+g}$ and the universal Picard variety. Consider an algebraic curve $C_g$ of genus $g \geq 1$ and a fixed point $p \in C_g$. For each positive integer $m$, there is a map from the symmetric product $\Sigma^m C_g$ to the Jacobian $\Jac(C_g)$ defined by sending the tuple of points $(x_1, \ldots, x_m)$ to the divisor $\sum_{i=1}^m x_i - m\cdot p$. Note that this morphism is not canonical, since it depends on the choice of $p \in C_g$. For the particular case of $m=g$, any such map defines a birational equivalence between $\Sigma^g C_g$ and $\Jac(C_g)$. For our context, this argument should be adapted to curves with marked points, excluding the diagonal, but we do not intend to study this relation in detail. We simply remark that each element of $\mathcal{M}_{g,n+g}$ can be seen as an element $(C_g; p_1, \ldots, p_n)$ of $\mathcal{M}_{g,n}$ equipped with $g$ extra distinct points. Fixing an extra point $p \in C_g$, these $g$ points may be used to determine a degree 0 line bundle on $C_g$.

We admit that~\cref{eq:ELSVdMgn} may seem unnatural, as it is possible to equivalently express the formula simply in terms of $\oM_{g,n}$. Expressing it in terms of $\oM_{g,n+g}$, however, allows one to match the powers of $d$ --- and therefore the degree in $\mu_1, \ldots, \mu_n$ --- with equal degree cohomology classes, as one might expect from the original ELSV formula. One presumes that the desire for an ELSV formula on a space of dimension $4g - 3 + n$ in the Goulden--Jackson--Vakil conjecture is mainly motivated by such degree considerations.

\subsection{Secondary properties} \label{subsec:secondary}

We summarise and briefly address other expected properties of an ELSV formula for one-part double Hurwitz numbers, collected from discussions throughout the paper of Goulden, Jackson and Vakil~\cite{gou-jac-vak}.

{\bf Property A.} {\em The classes $\Lambda_{2k} \in H^{4k}(\oPic_{g,n}) $ are ``tautological''.} (See~\cite[paragraph after Conjecture~3.5]{gou-jac-vak}.)

To interpret this statement, one would need to define the word ``tautological'' for any moduli space under consideration that is not a moduli spaces of curves. The proposed ELSV formulas of~\cref{eq:ELSVdMgn,eq:ELSVdMgng} use pushforwards of Chiodo classes to $\oM_{g,n}$ or $\oM_{g,n+g}$, and these classes are evidently tautological by Chiodo's formula for the Chern characters $\ch_k(R^* \pi_* \mathcal{L})$, given in~\cref{eq:chiodoformula}.

{\bf Property B.} {\em There exists a morphism $\rho: \oPic_{g,n} \to \oM_{g,n}$ such that $\rho_{*}\Lambda_{2g} = \lambda_g$.} (See~\cite[Conjecture 3.13]{gou-jac-vak}.)

We do not have such a relation but instead find Chiodo classes appearing in place of the Hodge classes of the original ELSV formula. The Chiodo classes appear with various parameters, but one does recover the usual Hodge class $\lambda_g$ as the degree $g$ part of $\epsilon_* \Ch_{g,n}(1,1;1, \ldots, 1)$.

{\bf Property C.} {\em The definition of the moduli space $\oPic_{g,n}$ and the associated $\Psi$-classes and $\Lambda$-classes should make evident the fact that string and dilaton contraints govern the intersection numbers of its $\Lambda$-classes.}

The string and dilaton equations are stated in~\cite[Proposition~3.10]{gou-jac-vak}, though we restate them below using the language of the present paper. Observe that the evaluations on the left sides of the equations use the polynomiality of the one-part double Hurwitz numbers.
\begin{align*}
\left[ h_{g; \mu_1, \ldots, \mu_n, \mu_{n+1}}^{\textnormal{one-part}} \right]_{\mu_{n+1}=0} &= (\mu_1 + \cdots + \mu_n) \, h_{g; \mu_1, \ldots, \mu_n}^{\textnormal{one-part}} \\
\left[ \frac{\partial}{\partial \mu_{n+1}} h_{g; \mu_1, \ldots, \mu_n, \mu_{n+1}}^{\textnormal{one-part}} \right]_{\mu_{n+1}=0} &= (2g-2+n) \, h_{g; \mu_1, \ldots, \mu_n}^{\textnormal{one-part}}
\end{align*}
It is not clear at present how these relate to the formulas of~\cref{thm:main}.

{\bf Property D.} {\em The low genus verifications in $g=0$ and in $g=1$ correspond with the spaces $\oPic_{0,n} = \oM_{0,n}$ and $\oPic_{1,n} = \oM_{1, n+1}$.} (See~\cite[Proposition~3.11]{gou-jac-vak}.)

Our fourth ELSV formula in~\cref{eq:ELSVdMgng} does indeed match this result.

{\bf Property E.} {\em The formula should make evident the fact that $h_{g; \mu_1, \ldots, \mu_n}^{\textnormal{one-part}}$ is a polynomial in $\mu_1, \ldots, \mu_n$. Moreover, it should be possible to deduce that this polynomial only contains monomials of degrees $2g-2+n$ to $4g - 2 + n$.}

Deducing the polynomiality of one-part double Hurwitz numbers from~\cref{thm:main} is not straightforward, as the dependence of the Chiodo classes on its parameters still remains for the most part rather mysterious. The best available result in this direction is the following, due to Janda, Pandharipande, Pixton and Zvonkine~\cite[Proposition~5]{JPPZ}: the degree $d$ part of the class $r^{2d-2g+1} \epsilon_*\Ch_{g,n}(r, s; a_1, \ldots, a_n)$ is polynomial in $r$ for sufficiently large $r$. Here, ``sufficiently large'' is meant with respect to the parameters $a_1, \ldots, a_n$, whereas in our case, $r = |\mu|$ is intrinsically linked to the parameters $\mu_1, \ldots, \mu_n$. Moreover, a computational and conceptual difficulty arises from the fact that Chiodo classes do not in general vanish in degree higher than $g$, unlike Hodge classes. Nevertheless, one might conceivably be able to prove the desired polynomiality of one-part double Hurwitz numbers via a careful analysis of the stable graph expression for Chiodo classes by Janda, Pandharipande, Pixton and Zvonkine~\cite[Corollary 4]{JPPZ}.

\subsection{Further remarks and corollaries of the main theorem}\label{sec:furtherremarks}

The formulas of~\cref{thm:main} do not provide an immediate explanation for the polynomiality of one-part double Hurwitz numbers, unlike the the original ELSV formula for single Hurwitz numbers. However, in the case of one-part double Hurwitz numbers, there is a clear combinatorial explanation for the polynomiality~\cite{gou-jac-vak}. So rather than seeking a geometric explanation, we propose that one should instead study the implications of polynomiality for the geometry of moduli spaces.

As mentioned in~\cref{sec:introduction}, Goulden, Jackson and Vakil proved that one-part double Hurwitz numbers satisfy
\begin{equation}\label{eq:gen1}
h_{g; \mu_1, \ldots, \mu_n}^{\textnormal{one-part}} = (\mu_1 + \cdots + \mu_n)^{2g-2+n} \, P_{g,n}(\mu_1^2, \ldots, \mu_n^2),
\end{equation}
for some symmetric polynomial $P_{g,n}$ of degree $g$~\cite{gou-jac-vak}, explicitly obtained as
\begin{equation}\label{eq:gen2}
P_{g,n}(\mu_1^2, \ldots, \mu_n^2) = [t^{2g}]. \frac{\prod_{i=1}^n \SSS(t \mu_i)}{\SSS(t)}
\end{equation} 
for $\SSS(x) = \sinh(x/2)/(x/2)$. Note that both $\SSS(x)$ and $1/\SSS(x)$ are holomorphic even functions near $x=0$. Combined with~\cref{thm:main}, this implies that 
\[
|\mu| \int_{\overline{\mathcal{M}}_{g,n}} \frac{\epsilon_* \Ch_{g,n}(|\mu|, |\mu|; -\mu_1, \ldots, -\mu_n)}{\prod_{i=1}^n (1 - \frac{\mu_i}{|\mu|} \psi_i)}
\]
is a polynomial in $\mu_1, \ldots, \mu_n$ of degree $2g$ that is moreover invariant under the symmetries $\mu_i \leftrightarrow -\mu_i$ for all $1 \leq i \leq n$. Let us remark that this behaviour is a priori unexpected. As mentioned earlier, the class $\epsilon_*\Ch_{g,n}(r, s; a_1, \ldots, a_n)$ exhibits certain polynomial behaviour in $r$ for sufficiently large $r$~\cite{JPPZ}. In our case, $r = |\mu|$ is intrinsically linked to the parameters $\mu_1, \ldots, \mu_n$ and the polynomiality we observe is actually in the parameters $\mu_1, \ldots, \mu_n$. The invariance $\mu_i \leftrightarrow -\mu_i$ is perhaps even more surprising. For instance, even considering the Chiodo class with the parameter $r$ set sufficiently large, this symmetry changes $a_i$ into $r-a_i$ in the class parameters. This in turn transforms the coefficient of the psi-class terms in Chiodo's formula for $\mathrm{ch}_k(r, s; a_1, \ldots, a_n)$ via $B_{k+1}(\frac{a_i}{r}) \mapsto B_{k+1}(\frac{r - a_i}{r}) = (-1)^{k+1} B_{k+1}(\frac{a_i}{r})$, thus introducing a change of sign for even degrees $k$ that must be compensated by the other summands of the formula in a non-trivial way.

Moreover, we can use the actual generating series for one-part Hurwitz numbers in terms of hyperbolic functions of formulae \eqref{eq:gen1} and \eqref{eq:gen2} in combination with our main \cref{thm:main} to obtain generating series for particular integrals of Chiodo classes. We do so for the particular cases of $\mu = (1^d)$ and of $\mu = d$.

For $\mu = (1^d)$, \cref{eq:gen2} turns into $P_{g,d}(1, \dots, 1) = [t^{2g}]. \SSS(t)^{d-1}$, and therefore by \cref{eq:gen1} and \cref{thm:main} we obtain the following evaluation.

\begin{proposition}\label{prop:genmu1}
For all positive integers $d$, for $2g - 2 + d > 0$ we have the following generating series for integrals of Chiodo classes:
\[
1 + d \cdot \sum_{g=1} t^{2g} \int_{\oM_{g,d}} \frac{\epsilon_* \Ch_{g,d}(d, d; -1, \ldots, -1)}{\prod_{i=1}^d(1 - \frac{\psi_i}{d})} = \SSS(t)^{d-1}.
\]
\end{proposition}

On the other hand, for $\mu = (d)$ and $d>1$, \cref{eq:gen2} turns into 
\begin{equation}\label{eq:muisd}
(2g)! P_{g,1}(d^2) = (2g)! [t^{2g}]. \frac{\SSS(dt )}{\SSS(t)} = \sum_{k = -\frac{d-1}{2}}^{\frac{d-1}{2}} k^{2g} = 
\begin{cases}
2 p_{2g}(1,2,3,\dots, \frac{d-1}{2}) &  \text{ for } d \text{ odd},
\\
 2^{1 - 2g} p_{2g}(1,3,5\dots, d-1)  & \text{ for } d \text{ even},
\end{cases}
\end{equation}
where $p_{2g}$ is the usual power sum of homogeneous degree $2g$. Let us recall the classical Faulhaber formula
\begin{equation}
\frac{1}{(2g)!}p_{2g}(1, 2, 3, \dots, N) = \sum_{k=0}^{2g} \frac{B^{+}_k}{k!} \frac{N^{2g +1 - k}}{(2g + 1 - N)!}, \qquad \qquad \sum_{m=0}^{\infty} \frac{B^{+}_m}{m!} x^m = \frac{x}{1 - e^{-x}}.
\end{equation}
Notice that the numbers $B_m^{+}$ are Bernoulli numbers which differ from the Bernoulli numbers $B_m$ we use throughout the paper just for the case $B_{1}^{+} = 1/2$ (instead of $B_m = -1/2$) and all the others coincide. Moreover, notice that Faulhaber formula proves that $p_{2g}(1, 2, 3, \dots, N)$ is manifestly a polynomial in $N$, and that moreover of degree $2g + 1$. In order to use Faulhaber formula for even $d$ , we can simply observe that 
$$
p_{2g}(1, 3, 5, \dots, 2N-1) =  p_{2g}(1, 2, 3, \dots, 2N) - 2^{2g} p_{2g}(1, 2, 3, \dots, N).
$$
For $d=1$ instead, we get $P_{2g}(1) = \delta_{g,0} = p_{2g}(0)$. Now, combining \cref{eq:muisd} with \cref{eq:gen1}, \cref{thm:main} for $\mu = (d)$, and the considerations above we obtain the following statement.

\begin{proposition}\label{prop:genmud}
For $g \geq 1$ we have the following evaluations of integrals of Chiodo classes:

For all odd positive integers $d > 1$ we have:
\begin{equation}
d \cdot \int_{\oM_{g,1}} \!\!\! \frac{\epsilon_* \Ch_{g,1}^{[d]}(d, d; d)}{(1 - \psi_1)} 
=
\frac{2}{(2g)!} p_{2g}\left(1,2,3,\dots, \frac{d-1}{2}\right)
=
 \sum_{k=0}^{2g} \frac{B_k^+}{k!}\frac{(d-1)^{2g+1-k}}{2^{2g - k} (2g + 1 - k)!},
\end{equation}

for all even positive integers $d$ we have:
\begin{equation}
d \cdot \int_{\oM_{g,1}} \!\!\! \frac{\epsilon_* \Ch_{g,1}^{[d]}(d, d; d)}{(1 - \psi_1)} 
=
\frac{2^{1-2g}}{(2g)!} p_{2g}(1,3,5,\dots, d-1)
=
 \sum_{k=0}^{2g} \frac{B_k^+}{k!}\frac{d^{2g+1-k}}{ (2g + 1 - k)!}\left[\frac{1 - 2^{k-1}}{2^{2g -1}}\right],
\end{equation}
for $p_{2k}$ the power sum and $B_k^+$ the Bernoulli numbers with the convention defined above. For $d=1$ the first statement for odd $d$ actually still holds with the convention that 
\end{proposition}

The discussion above for $d=1$ instead gives the well-known
$
 \int_{\oM_{g,1}} \!\!\! \frac{\Lambda(-1)}{(1 - \psi_1)} 
=
0
$ for all $g \geq 1$ (notice for instance that \cref{prop:genmu1} for $d=1$ recovers the same result, but it can be immediately obtained from ELSV formula).

\section{Comparing ELSV formulas past and present} \label{sec:comparison}

In this section, we compare the original ELSV formula in~\cref{eq:ELSV} for single Hurwitz numbers and our new ELSV formula in~\cref{eq:ELSVdMgn} for one-part double Hurwitz numbers. Specialising these formulas in a particular way results in rather non-trivial relations between tautological intersection numbers.

\subsection{Exchanging ramification profiles} \label{subsec:exchange}

Let us begin by specialising the following two ELSV formulas.
\begin{enumerate} [label=(\roman*)]
\item First, consider the original ELSV formula in~\cref{eq:ELSV} for single Hurwitz numbers, with the tuple $(\mu_1, \ldots, \mu_n)$ set to the 1-tuple $(d)$.
\[
h_{g;d} = \frac{d^d}{d!} \int_{\oM_{g,1}} \frac{\sum_{k=0}^g (-1)^k \lambda_{k}}{1 - d \psi_1}
\]

\item Second, consider the new ELSV formula in~\cref{eq:ELSVdMgn} for one-part double Hurwitz numbers, with the tuple $(\mu_1, \ldots, \mu_n)$ set to the $d$-tuple $(1, 1, \ldots, 1)$.
\[
h_{g;1, 1, \ldots, 1}^{\textnormal{one-part}} = d^{2-g} \int_{\oM_{g,d}} \frac{\epsilon_* \Ch_{g,d}^{[d]}(d, d; -1, \ldots, -1)}{\prod_{i=1}^d (1 - \psi_i)}
\]
\end{enumerate}

Now we simply note the equality
\[
h_{g;d} = \frac{1}{d!} \, h_{g;1, 1, \ldots, 1}^{\textnormal{one-part}},
\]
which uses the symmetry that exchanges ramification profiles over $0$ and $\infty$. In other words, both sides enumerate genus $g$ branched covers of $\mathbb{CP}^1$ with ramification $(d)$ over one point, ramification profile $(1, 1, \ldots, 1)$ over another point, and simple branching elsewhere. The factor of $\frac{1}{d!}$ on the right side of the equation is to compensate for the factor $|\mathrm{Aut}(\mu)|$ that appears in the weighting attached to a branched cover, given in~\cref{eq:hurwitzweight}. This allows us to compare the two ELSV formulas and obtain the following relation.

\begin{theorem} \label{thm:comparison}
For integers $g \geq 0$ and $d \geq 1$ with $(g,d) \neq (0,1)$ or $(0,2)$, we have
\begin{equation} \label{eq:relationELSV}
\int_{\oM_{g,1}} \frac{\sum_{k=0}^g (-1)^k \lambda_{k}}{1 - d \psi_1}
= \frac{1}{d^{d + g - 2}} \int_{\oM_{g,d}} \frac{\epsilon_* \Ch_{g,d}^{[d]}(d, d; -1, \ldots, -1)}{\prod_{i=1}^d(1 - \psi_i)}.
\end{equation}
In the case $g = 0$, the left side is to be interpreted via the usual convention: $ \int_{\oM_{0,1}}\frac{1}{1 - d\psi_1} = \frac{1}{d^2}$.
\end{theorem}

Observe that the left side of~\cref{eq:relationELSV} is inherently polynomial in $d$, while the right side is not. At present, it is not clear how to argue that the right side is polynomial in $d$ without invoking~\cref{thm:comparison} itself.

\subsection{Another proof of \cref{prop:genmu1} }
Finally, we recall a result of Faber and Pandharipande~\cite[Theorem~2]{fab-pan}, which can be expressed as
\[
1 + \sum_{g=1}^\infty d^2 t^{2g} \int_{\oM_{g,1}} \frac{\sum_{k=0}^g (-1)^k \lambda_k}{1 - d\psi_1} = \SSS(dt)^{d-1}, 
\qquad \qquad 
\SSS(x) = \frac{\sinh(x/2)}{x/2}.
\] 
Notice that combining this with~\cref{thm:comparison} immediately produces a generating series for the integrals appearing on the right side of~\cref{eq:relationELSV}:
\[
1 + \sum_{g=1}^\infty t^{2g}  \frac{1}{d^{d + g - 4}} \int_{\oM_{g,d}} \frac{\epsilon_* \Ch_{g,d}^{[d]}(d, d; -1, \ldots, -1)}{\prod_{i=1}^d(1 - \psi_i)} = \SSS(dt)^{d-1}.
\] 

This obtained generating series is equivalent to the statement of \cref{prop:genmu1} after the normalisation of the variable $t \mapsto t/d$ and the extraction of $3g - 3 + d$ powers of $d$, and therefore provides another proof of it via Faber-Pandharipande theorem. 

This is not surprising at all: in fact Goulden-Jackson-Vakil polynomiality result of \cref{eq:gen1} and \cref{eq:gen2} provides a generalisation of Faber-Pandharipande theorem (restrict to it, as expected, for $\mu = (1^d)$), as it can be seen applying the classical ELSV formula. However, the theorem of Faber and Pandharipande was derived (shorty) before ELSV formula and hence also before Goulden-Jackson-Vakil polynomiality: its proof relies on independent methods and therefore a different proof of \cref{prop:genmu1} can be obtained this way.

\subsection{Generalisation to the double orbifold case} \label{subsec:orbifold}

The strategy used to derive~\cref{thm:comparison} may be pushed further to yield a more general statement involving $q$-orbifold Hurwitz numbers. Let us proceed by specialising the following two ELSV formulas.
\begin{enumerate} [label=(\roman*)]
\item First, consider the ELSV formula in~\cref{eq:JPTalternative} for $q$-orbifold Hurwitz numbers, with the tuple $(\mu_1, \ldots, \mu_n)$ set to the 1-tuple $(d)$, where $d$ is a multiple of $q$.
\[
h_{g;d}^{q\textnormal{-orbifold}} = q^{2g-1} \frac{d^{d/q}}{(d/q)!} \int_{\overline{\mathcal M}_{g,1}} \frac{\epsilon_* \Ch_{g,1}(q, q; q)}{\big(1 - \frac{d}{q} \psi_i \big)}
\]

\item Second, consider the new ELSV formula in~\cref{eq:ELSVdMgn} for one-part double Hurwitz, with the tuple $(\mu_1, \ldots, \mu_n)$ set to the tuple $(q, q, \ldots, q)$ with $\frac{d}{q}$ parts.
\[
h_{g; q, q, \ldots, q}^{\textnormal{one-part}} = d^{2-g} \int_{\overline{\mathcal{M}}_{g,d/q}} \frac{\epsilon_* \Ch_{g,d/q}^{[d]}(d, d; -q, \ldots, -q)}{\prod_{i=1}^{d/q} (1 - q \psi_i)}
\]
\end{enumerate}

Now we simply note the equality
\[
h_{g;d}^{q\textnormal{-orbifold}} = \frac{1}{(d/q)!} \, h_{g;q, q, \ldots, q}^{\textnormal{one-part}},
\]
which again uses the symmetry that exchanges ramification profiles over $0$ and $\infty$. This allows us to compare the two ELSV formulas and obtain the following relation.
\[
q^{2g-1} d^{d/q} \int_{\overline{\mathcal M}_{g,1}} \frac{\epsilon_* \Ch_{g,1}(q, q; q)}{1 - \frac{d}{q} \psi_i} = d^{2-g} \int_{\overline{\mathcal{M}}_{g,d/q}} \frac{\epsilon_* \Ch_{g,d/q}^{[d]}(d, d; -q, \ldots, -q)}{\prod_{i=1}^{d/q} (1 - q \psi_i)}
\]
Performing the substitutions $q = a$ and $d = a b$ leads to the following result.

\begin{theorem} \label{thm:orbifoldcomparison}
For integers $g \geq 0$ and $a, b \geq 1$ with $(g,b) \neq (0,1)$ or $(0,2)$, we have
\[
\int_{\overline{\mathcal M}_{g,1}} \frac{\epsilon_* \Ch_{g,1}(a, a; a)}{1 - b \psi_i} = \frac{1}{b^{b+g-2}} \int_{\overline{\mathcal{M}}_{g,b}} \frac{\epsilon_* \Ch_{g,n}^{[b]}(a b, a b; -a, \ldots, -a)}{\prod_{i=1}^b (1-\psi_i)}.
\]
\end{theorem}

Following a suggestion of Sergey Shadrin, we can push the strategy used to derive~\cref{thm:comparison,thm:orbifoldcomparison} even further. Consider positive integers $p$ and $q$, as well as a positive integer $d$ that is a multiple of them both. The symmetry that exchanges ramification profiles over $0$ and $\infty$ leads to the following equality of degree $d$ orbifold Hurwitz numbers.
\[
\frac{h_{g; p, p, \ldots, p}^{q\textnormal{-orbifold}}}{(d/p)!} = \frac{h_{g; q, q, \ldots, q}^{p\textnormal{-orbifold}}}{(d/q)!}
\]
Each side of this equation may be expressed as a specialisation of the the orbifold ELSV formula in~\cref{eq:JPTalternative} to obtain the following generalisation of~\cref{thm:orbifoldcomparison}.

\begin{theorem} \label{thm:doubleorbifoldcomparison}
Let $p < q$ be positive integers and $d$ a multiple of them both. For integers $g \geq 0$ and $d \geq 1$ with $(g, d/p), (g, d/q) \neq (0,1)$ or $(0,2)$, we have
\begin{multline*}
\frac{1}{(d/p)! \, p^{2g-2+\frac{d}{p}+\frac{d}{q}}} \int_{\overline{\mathcal M}_{g,d/p}} \frac{\epsilon_* \mathrm{Chiodo}(q, q; -p, \ldots, -p)}{\prod_{i=1}^{d/p} (1 - \frac{p}{q} \psi_i)} \\
= \frac{1}{(d/q)! \, q^{2g-2+\frac{d}{p}+\frac{d}{q}}} \left( \frac{(q/p)^{\lfloor q/p \rfloor}}{\lfloor q/p \rfloor!} \right)^{d/q} \int_{\overline{\mathcal M}_{g,d/q}} \frac{\epsilon_* \mathrm{Chiodo}(p, p; -q, \ldots, -q)}{\prod_{i=1}^{d/q} (1 - \frac{q}{p} \psi_i)}.
\end{multline*}
\end{theorem}

\section{Verification in low genus} \label{sec:verification}

The goal of this section is to check by hand some of the first cases of the relations between tautological numbers obtained in the previous section. We will focus on~\cref{thm:comparison}, computing explicitly both sides for genus zero and all $d$, as well as for genus one and $d=1$. Also, higher $d$ cases have been tested by means of the Sage package \textit{admcycles}, see \cite{adm}.

\subsubsection*{Genus 0}

Let us verify~\cref{eq:relationELSV} in genus zero. Recall that, for the unstable cases $(g,n) = (0,1)$ and $(0,2)$, integrals of weighted psi-classes are taken by the usual convention to be
\[
 \int_{\oM_{0,1}}\frac{1}{1 - x \psi_1} = \frac{1}{x^2} \qquad \text{and} \qquad \int_{\oM_{0,2}}\frac{1}{(1 - x \psi_1)(1 - y \psi_2)} = \frac{1}{x+y}.
\]
In genus zero, the Hodge bundle $\Lambda$ is trivial, so~\cref{eq:relationELSV} reads
\begin{equation}\label{eq:g0rel}
\int_{\oM_{0,1}} \frac{1}{1 - d \psi_1} = \frac{1}{d^{d - 2}} \int_{\oM_{0,d}} \frac{\epsilon_{*} \Ch_{0,d}^{[d]}(d, d; -1, \ldots, -1)}{\prod_{i=1}^d(1 - \psi_i)}.
\end{equation}

The left side of~\cref{eq:g0rel} is simply $\frac{1}{d^2}$, using the unstable calculation mentioned previously. So let us now compute the right side for $d \geq 3$. In genus zero, the geometric situation is rather simple. Even though the Chiodo class is defined a priori as 
\[
c( - R^*\pi_*\mathcal{L}) = c( - R^0\pi_*\mathcal{L} + R^1\pi_*\mathcal{L} ),
\] 
the result of~\cite[Proposition~4.4]{jar-kim-vai} guarantees that $R^0\pi_*\mathcal{L}$ vanishes in genus zero. Hence, the Chiodo class becomes the Chern class of an actual vector bundle, and thus, vanishes in degree higher than its rank. For the general Chiodo class $\Ch_{g,n}(r, s; a_1, \ldots, a_n)$, the Riemann--Roch theorem for line bundles gives
\[
\frac{(2g - 2 + n)s - \sum_{i=1}^n a_i}{r} - g + 1 = h^0 - h^1.
\]
After substituting $g=0$, $r = s = d = n$, $a_i = d-1$, and setting $h^0 = 0$, we find that the rank is equal to 
\[
\frac{(d - 2)d - d(d-1)}{d} + 1= 0.
\]
Therefore, the Chiodo class in this case contributes only in degree zero, so it must be equal to 1. The pushforward then produces a global factor of $d^{-1}$. Therefore, the right side of~\cref{eq:g0rel} is equal to 
\[
\frac{1}{d^{d-2}} \cdot \frac{1}{d} \int_{\oM_{0,d}} \frac{1}{\prod_{i=1}^d(1 - \psi_i)} = \frac{1}{d^{d-1}} \sum_{a_1 + \cdots + a_d = d-3} \binom{d-3}{a_1, \ldots, a_d} = \frac{1}{d^{d-1}} \, d^{d-3} = \frac{1}{d^2},
\]
which completes the verification. Note that we have used here the well-known formula for psi-class intersection numbers in genus zero~\cite{har-mor}.

\subsubsection*{Genus 1}

Let us consider~\cref{eq:relationELSV} in genus one, which can be expressed as
\begin{equation}\label{eq:relg1}
d^{d-1} \int_{\oM_{1,1}} \frac{1 - \lambda_1}{1 - d \psi_1} = \int_{\oM_{1,d}} \frac{\epsilon_* \Ch_{1,d}^{[d]}(d, d; -1, \ldots, -1)}{\prod_{i=1}^d(1 - \psi_i) }.
\end{equation}
The left side is immediately computed as $d^{d-1} \cdot \frac{d-1}{24}$.

The calculation of the right side of~\cref{eq:relg1} for general $d$ is computationally intensive. One difficulty lies in the fact that, whereas the Hodge class vanishes in degree higher than the genus, such vanishing for the Chiodo class cannot be guaranteed. At present, it is unclear whether such vanishing may arise in our case, in which the Chiodo class parameters are tuned in a particular way.

We will proceed by computing the contribution of the degree zero and degree one terms from the Chiodo class. These are the only terms that contribute to the right side of~\cref{eq:relg1} in the case $d = 1$, since $\oM_{1,1}$ has dimension one. This will allow us to check~\cref{eq:relg1} only in the case $d = 1$ and may convince the reader that further checks require significant computation or new ideas.

Our computations will rely on the following result~\cite[Theorem~2.3]{ELSV}.

\begin{lemma} \label{lem:g1psi}
For $\mu_1 + \cdots + \mu_n = d$, we have
\[
\int_{\oM_{1, n}} \frac{1}{\prod_{i=1}^{n} (1 - \mu_i \psi_i)} = \frac{1}{24} \Bigg[ d^n - \sum_{j=2}^{n} (j-2)! \, d^{n - j} \, e_j(\mu_1, \ldots, \mu_{n}) \Bigg],
\]
where $e_j$ denotes the $j$th elementary symmetric polynomial.
\end{lemma}

\subsubsection*{Right side of~\cref{eq:relg1}: degree zero} 

The degree zero part of the Chiodo class is simply equal to 1, although a global factor of $d$ arises from the pushforward to the moduli space of curves. So the contribution of the Chiodo class in degree zero can be calculated using~\cref{lem:g1psi} to obtain
\begin{equation} \label{eq:summand0}
d \int_{\oM_{1,d}} \frac{1}{\prod_{i=1}^d(1 - \psi_i)} = \frac{d}{24} \Bigg[ d^d - \sum_{j=2}^{d} (j-2)! \, d^{d-j} \, \binom{d}{j} \Bigg].
\end{equation}

\subsubsection*{Right side of~\cref{eq:relg1}: degree one}

Chiodo's formula of~\cref{eq:chiodoformula} asserts that the first Chern character associated to the Chiodo class of~\cref{eq:relg1} equals
\begin{align*}
\ch_1(d, d; d-1, \ldots, d-1) &= \frac{1}{2!} \left[ B_2(1) \, \kappa_1 - \sum_{i=1}^d B_2(\tfrac{d-1}{d}) \, \psi_i + \frac{d}{2} \sum_{a=0}^{d-1} B_2(\tfrac{a}{d}) \, {j_a}_* 1 \right] \\
&= \frac{1}{12} \kappa_1 - \sum_{i=1}^d \frac{d^2 - 6d + 6}{12 \, d^2} \, \psi_i + \frac{d}{4} \sum_{a=0}^{d-1} \frac{d^2 - 6ad + 6a^2}{6 d^2} \, {j_a}_* 1.
\end{align*}
This allows us to express the degree one contribution of the right side of~\cref{eq:relg1} as follows.
\[
- \frac{d}{12} \int_{\oM_{1,d}} \frac{\epsilon_* \kappa_1}{\prod_{i=1}^d(1 - \psi_i)} + \frac{d^2-6d+6}{12} \int_{\oM_{1,d}} \frac{\epsilon_* \psi_1}{\prod_{i=1}^d(1 - \psi_i)} - \frac{1}{24} \sum_{a=0}^{d-1} (d^2-6ad+6a^2) \int_{\oM_{1,d}} \frac{\epsilon_* {j_a}_* 1}{\prod_{i=1}^d(1 - \psi_i)}
\]

We now proceed to compute each of the three summands above separately. The evaluation of the first summand amounts to compute
\[
- \frac{d^2}{12} \int_{\oM_{1,d}} \frac{ \kappa_1}{\prod_{i=1}^d(1 - \psi_i)} = - \frac{d^2}{12} \sum_{a_1 + \cdots + a_d = d-1} \int_{\oM_{1,d+1}} \psi_1^{a_1} \cdots \psi_d^{a_d} \psi_{d+1}^2.
\]
Note that an extra factor of $d$ arises from the pushforward of the class $\kappa_1$ from the moduli space of spin curves. Using~\cref{lem:g1psi}, we calculate the first summand as follows.
\begin{align}
& - \frac{d^2}{12} \sum_{a_1 + \cdots + a_d = d-1} \int_{\oM_{1,d+1}} \psi_1^{a_1} \cdots \psi_d^{a_d} \psi_{d+1}^2 \notag \\
={}& \frac{-d^2}{12 \cdot 24} \left[ \mu_{d+1}^2 \right] \Bigg[ \bigg( \sum_{i=1}^d \mu_i + \mu_{d+1} \bigg)^{d+1} - \sum_{j=2}^{d+1} (j-2)! \, \bigg( \sum_{i=1}^d \mu_i + \mu_{d+1} \bigg)^{d+1 - j} e_j(\mu_1, \ldots, \mu_{d+1}) \Bigg]_{\mu_1 = \cdots = \mu_d = 1} \notag \\
={}& \frac{-d^2}{12 \cdot 24} \Bigg[ \binom{d+1}{2} \bigg(\sum_{i=1}^d \mu_i \bigg)^{d-1} - \sum_{j=2}^{d+1} (j-2)! \, \Bigg[ \binom{d+1-j}{2} \bigg( \sum_{i=1}^d \mu_i \bigg)^{d - 1 - j} e_j(\mu_1, \ldots, \mu_{d}) \notag \\
& \qquad \qquad \quad + \binom{d+1-j}{1} \bigg( \sum_{i=1}^d \mu_i \bigg)^{d - j} e_{j-1}(\mu_1, \ldots, \mu_{d}) \Bigg] \Bigg]_{\mu_1 = \cdots = \mu_d = 1} \notag \\
={}& \frac{-d^2}{12 \cdot 24} \Bigg[ \binom{d+1}{2}d^{d-1} - \sum_{j=2}^{d+1} (j-2)! \, \bigg[ \binom{d+1-j}{2} d^{d-1 - j} \binom{d}{j} + \binom{d+1-j}{1} d^{d - j} \binom{d}{j-1} \bigg] \Bigg] \label{eq:summand1}
\end{align}

The evaluation of the second summand can be written as follows, using the fact that the pushforward produces a factor of $d^{1}$.
\begin{align*}
\frac{d^2 - 6d + 6}{12} \int_{\oM_{1,d}} \frac{\epsilon_* \psi_1}{\prod_{i=1}^d(1 - \psi_i)} &= \frac{d(d^2 - 6d + 6)}{12} \int_{\oM_{1,d}} \frac{ \psi_1}{\prod_{i=1}^d(1 - \psi_i)} \\
&= \frac{d(d^2 - 6d + 6)}{12} \Bigg[ \int_{\oM_{1,d}} \frac{1}{\prod_{i=1}^d(1 - \psi_i)} - \int_{\oM_{1,d}} \frac{1}{\prod_{i=2}^d(1 - \psi_i)} \Bigg]
\end{align*}
Setting all $\mu_i = 1$ in~\cref{lem:g1psi}, we get
\[
\int_{\oM_{1,d}} \frac{1}{\prod_{i=1}^d (1 - \psi_i)} = \frac{1}{24} \Bigg[ d^d - \sum_{j=2}^{d} (j-2)! \, d^{d-j} \binom{d}{j} \Bigg],
\]
and similarly, setting all $\mu_i = 1$ with the exception of $\mu_1 = 0$, we have
\[
\int_{\oM_{1,d}} \frac{1}{\prod_{i=2}^d (1 - \psi_i)} = \frac{1}{24} \Bigg[ (d-1)^d - \sum_{j=2}^{d} (j-2)! \, (d-1)^{d-j} \binom{d-1}{j} \Bigg].
\]
Therefore, the second summand of the degree one contribution is
\begin{equation} \label{eq:summand2}
\frac{d(d^2 - 6d + 6)}{12 \cdot 24} \Bigg[ d^d - (d-1)^d - \sum_{j=2}^{d} (j-2)! \bigg[ d^{d-j} \binom{d}{j} - (d-1)^{d-j} \binom{d-1}{j} \bigg] \Bigg].
\end{equation}

For the third summand, we adopt the language of stable graphs, as per the Janda--Pandharipande--Pixton--Zvonkine formula for the pushforward of the Chiodo class \cite[Corollary~4]{JPPZ}; we refer the reader to their paper for further details. The calculation of the class $\epsilon_*\ch_1(d, d; d-1, \ldots, d-1)$ on $\overline{\mathcal M}_{1, d}$ requires contributions from the two types of stable graphs pictured below. In both cases, the number of leaves in total is $d$ and the leaves are labelled with the Chiodo class parameter $d-1$. The label $a \in \{0, 1, \ldots, d-1\}$ on the half-edge matches the multiplicity index $a$ appearing in Chiodo's formula and this forces the incident half-edge to be labelled by $d-a$, according to the local edge condition. The pushforward of a Chiodo class to $\overline{\mathcal M}_{g,n}$ introduces a factor of $d^{2g-1-h^1(\Gamma)}$ for the stable graph $\Gamma$, where $h^1$ denotes the first Betti number. Thus, we obtain the extra contribution of $d^1$ from the stable graph on the left and $d^0$ for the stable graph on the right.

\begin{figure} [ht!]
\centering
\begin{tikzpicture}
\draw (0,0) -- (3,0);
\draw (0,0) -- (-1,1);
\draw (0,0) -- (-1,-1);
\draw (0,0) -- (-1.3660,0.3660);

\draw (3,0) -- (4,1);
\draw (3,0) -- (4,-1);
\draw (3,0) -- (4.3660,0.3660);

\filldraw (-1.0607,-0.0000) circle (0.02);
\filldraw (-1.0245,-0.2745) circle (0.02);
\filldraw (-0.9186,-0.5303) circle (0.02);

\filldraw (4.0607,-0.0000) circle (0.02);
\filldraw (4.0245,-0.2745) circle (0.02);
\filldraw (3.9186,-0.5303) circle (0.02);

\filldraw[draw=black,fill=white] (0,0) circle (0.4);
\filldraw[draw=black,fill=white] (3,0) circle (0.4);
\node at (0,0) {1};
\node at (3,0) {0};

\node[left] at (-1,1) {$d-1$};
\node[left] at (-1.3660,0.3660) {$d-1$};
\node[left] at (-1,-1) {$d-1$};

\node[right] at (4,1) {$d-1$};
\node[right] at (4.3660,0.3660) {$d-1$};
\node[right] at (4,-1) {$d-1$};

\node[above right] at (0.4,0) {$a$};
\node[above left] at (2.6,0) {$d-a$};

\begin{scope}[xshift=10cm]
\draw (0,0) -- (-1,1);
\draw (0,0) -- (-1,-1);
\draw (0,0) -- (-1.3660,0.3660);

\filldraw (-1.0607,-0.0000) circle (0.02);
\filldraw (-1.0245,-0.2745) circle (0.02);
\filldraw (-0.9186,-0.5303) circle (0.02);

\draw plot [smooth,tension=1] coordinates {(0,0) (0.85,0.5) (1.5,0) (0.85,-0.5) (0,0)};
\filldraw[draw=black,fill=white] (0,0) circle (0.4);
\node at (0,0) {0};

\node[left] at (-1,1) {$d-1$};
\node[left] at (-1.3660,0.3660) {$d-1$};
\node[left] at (-1,-1) {$d-1$};

\node[above] at (0.85,0.5) {$a$};
\node[below] at (0.85,-0.5) {$d-a$};
\end{scope}
\end{tikzpicture}
\end{figure}

Let us proceed by analysing the contribution of the stable graph on the left. The local vertex condition requires that the sum of the labels of the half-edges adjacent to any given vertex is 0 modulo $d$. This imposes the constraint that the number of leaves on the genus 1 vertex is $a$, so the number of leaves on the genus 0 vertex is $d-a$. The stability condition requires the genus 0 vertex to have valence at least 2, which then rules out $a = d-1$.

Note that, for each fixed $a$, this stable graph can be obtained in $\binom{d}{a}$ ways, which correspond to the choices of markings of the leaves on the genus one vertex. Multiplying by the factor $d^{1}$ from the pushforward, the contribution resulting from this type of stable graph is as follows.
\begin{multline} \label{eq:summand3a}
- \frac{d}{24} \sum_{a=0}^{d-2} (d^2 - 6ad + 6a^2) \binom{d}{a}\left( \int_{\oM_{1, a+1} } \frac{\psi_{a+1}^0}{\prod_{i=1}^a (1 - \psi_i)} \right) \left( \int_{\oM_{0, d-a+1} } \frac{\psi_{d-a+1}^0 }{\prod_{i=1}^{d-a} (1 - \psi_i)} \right) \\
= -\frac{d}{24} \sum_{a=0}^{d-2} (d^2 - 6ad + 6a^2) \binom{d}{a} \Bigg[ a^{a+1} - \sum_{j=2}^{a+1} (j-2)! \, a^{a+1-j} \binom{a}{j} \Bigg] (d-a)^{d-a-2}
\end{multline}

Now let us analyse the contribution of the stable graph on the right. Observe that any labelling of the half-edges with $a$ and $d-a$ automatically satisfies the local vertex condition. As previously mentioned, the pushforward produces a factor of $d^0$ in this case. The integration pulls back to the space $\oM_{0, d+2}$, without psi-classes attached to the branches of the desingularized node, so the contribution resulting from this type of stable graph is as follows.
\begin{equation} \label{eq:summand3b}
- \frac{1}{24} \sum_{a=0}^{d-1} (d^2 - 6ad + 6a^2) \left( \int_{\oM_{0, d+2} } \frac{\psi_{d+1}^0 \psi_{d+2}^0}{\prod_{i=1}^d (1 - \psi_i)} \right) = - \frac{1}{24} \sum_{a=0}^{d-1} \left( d^2 - 6ad + 6a^2\right) d^{d-1} = - \frac{1}{24} d^d
\end{equation}

Finally, the total degree zero and one contribution on the right side of~\cref{eq:relg1} is obtained by adding the results of~\cref{eq:summand0,eq:summand1,eq:summand2,eq:summand3a,eq:summand3b}, and we obtain the following.
\begin{align}\label{eq:fourlines}
&\frac{d}{24} \Bigg[ d^d - \sum_{j=2}^{d} (j-2)! \, d^{d-j} \binom{d}{j} \Bigg] \notag \\
&-\frac{d^2}{12 \cdot 24}
\Bigg[ \binom{d+1}{2}d^{d-1} - \sum_{j=2}^{d+1} (j-2)! \, \bigg[ \binom{d+1-j}{2} d^{d-1 - j} \binom{d}{j} + \binom{d+1-j}{1} d^{d - j} \binom{d}{j-1} \bigg] \Bigg] \notag \\
&+ \frac{d(d^2 - 6d + 6)}{12 \cdot 24} \Bigg[ d^d - (d-1)^d - \sum_{j=2}^{d} (j-2)! \bigg[ d^{d-j} \binom{d}{j} - (d-1)^{d-j} \binom{d-1}{j} \bigg] \Bigg] \notag \\
&-\frac{d}{24} \sum_{a=0}^{d-2} (d^2 - 6ad + 6a^2) \binom{d}{a} \Bigg[ a^{a+1} - \sum_{j=2}^{a+1} (j-2)! \, a^{a+1-j} \binom{a}{j} \Bigg] (d-a)^{d-a-2} - \frac{1}{24} d^d.
\end{align}

\subsubsection*{Verification for $d=1$} 

Specialising the four lines of~\cref{eq:fourlines} to the case $d=1$, we obtain
\[
\frac{1}{24} - \frac{1}{12 \cdot 24} + \frac{1}{12 \cdot 24} - \frac{1}{24} = 0.
\]
On the other hand, we earlier computed the left side of~\cref{eq:relg1} to be $d^{d-1} \cdot \frac{d-1}{24}$, which also vanishes for $d=1$, as expected. The computation for $d \geq 2$ a priori requires higher degree calculations of the Chiodo class, which in turn requires many more stable graph contributions.

\section{Generalisation to the spin case} \label{sec:spin}

Some of the work of Goulden, Jackson and Vakil~\cite{gou-jac-vak} was generalised to the spin Hurwitz setting by Shadrin, Spitz and Zvonkine~\cite{SSZ}, who deduce polynomiality for one-part spin double Hurwitz numbers and conjecture an ELSV-type formula. In this section, we state generalisations of our earlier results -- namely,~\cref{thm:main,thm:orbifoldcomparison} --- to the spin setting, thus addressing the conjecture of Shadrin, Spitz and Zvonkine.

The double Hurwitz number $h_{g; \mu, \nu}$ may be interpreted as a relative Gromov--Witten invariant of $\mathbb{CP}^1$, in which the simple branch points correspond to insertions of $\tau_1$. For a positive integer $r$, one can more generally define the $r$-spin Hurwitz number analogously as a relative Gromov---Witten invariant of $\mathbb{CP}^1$, where the branching away from $0$ and $\infty$ corresponds to insertions of $\tau_r$. This is described in the work of Okounkov and Pandharipande on the Gromov--Witten/Hurwitz correspondence~\cite{oko-pan}, as well as in the work of Shadrin, Spitz and Zvonkine in their work on double Hurwitz numbers with completed cycles~\cite{SSZ}.

We focus on the $q$-orbifold $r$-spin Hurwitz numbers, which may be defined as the following relative Gromov--Witten invariants of $\mathbb{CP}^1$.
\[
h_{g; \mu_1, \ldots, \mu_n}^{q\textnormal{-orbifold}, r\textnormal{-spin}} = \frac{(r!)^m}{m!} \int_{[\overline{\mathcal M}_{g, m}(\mathbb{CP}^1; \mu, (q, \ldots, q))]^\text{vir}} \mathrm{ev}_1^*(\omega) \psi_1^r \cdots \mathrm{ev}_m^*(\omega) \psi_m^r
\]
Here, $\overline{\mathcal M}_{g, m}(\mathbb{CP}^1; \mu, (q, \ldots, q))$ denotes the space of stable genus $g$ maps to $\mathbb{CP}^1$ relative to 0 and $\infty$ with respective profiles $\mu$ and $(q, q, \ldots, q)$, with $m$ marked points where $m = \frac{1}{r}( 2g-2+n+|\mu|/q)$. As usual, $\mathrm{ev}_i$ is the evaluation map and we integrate over the virtual fundamental class.

Morally, this counts connected genus $g$ branched covers of $\mathbb{CP}^1$ with ramification profile $\mu$ over $\infty$, ramification profile $(q, q, \ldots, q)$ over $0$, and order $r$ branching elsewhere. The $q$-orbifold Hurwitz numbers are recovered when $r = 1$ and the single Hurwitz numbers thereafter by setting $q = 1$. The definition of the Gromov--Witten invariant takes into account stable maps, in which the domain curve may be nodal and components can map with degree zero. At the level of monodromy representations, such invariants may be described elegantly via factorisations into completed cycles in the symmetric group. This has been thoroughly described by Okounkov and Pandharipande~\cite{oko-pan}.

Zvonkine conjectured a polynomial structure for spin Hurwitz numbers akin to that for single Hurwitz numbers~\cite{zvo}. He furthermore posited an ELSV formula expressing spin Hurwitz numbers as intersection numbers on moduli spaces of spin curves. Kramer, Popolitov, Shadrin and the second author expressed a more general conjecture for orbifold spin Hurwitz numbers~\cite{KLPS}. This stronger conjecture has now been proved by Dunin-Barkowski, Kramer, Popolitov and Shadrin~\cite{DKPS}.

\begin{theorem} [Zvonkine's $q$-orbifold $r$-spin ELSV formula]
Fix positive integers $q$ and $r$. For integers $g \geq 0$ and $n \geq 1$ with $(g,n) \neq (0,1)$ or $(0,2)$, the $q$-orbifold $r$-spin Hurwitz numbers satisfy
\[
h_{g; \mu_1, \ldots, \mu_n}^{q\textnormal{-orbifold}, r\textnormal{-spin}} = r^{2g-2+n} (qr)^{\frac{(2g-2+n)q+\sum_{i=1}^n \mu_i}{qr}} \prod_{i=1}^n \frac{(\mu_i/qr)^{\lfloor \mu_i /qr \rfloor}}{\lfloor \mu_i / qr \rfloor!} \int_{\overline{\mathcal M}_{g,n}} \frac{\epsilon_* \Ch_{g,n}(qr; q; -\mu_1, \ldots, -\mu_n)}{\prod_{i=1}^n (1-\frac{\mu_i}{qr} \psi_i)}.
\]
\end{theorem}

The work of Goulden, Jackson and Vakil for double Hurwitz numbers was generalised by Shadrin, Spitz and Zvonkine~\cite{SSZ} to derive a polynomial structure for double spin Hurwitz numbers. It is then natural to define one-part double spin Hurwitz numbers as follows.

\begin{definition}
Let $h_{g; \mu}^{(r), \textnormal{one-part}}$ denote the double spin Hurwitz number $h_{g; \mu; \nu}^{r\textnormal{-spin}}$, where $\nu$ is the partition with precisely one part, which is equal to $|\mu|$.
\end{definition}

The polynomiality derived by Shadrin, Spitz and Zvonkine then leads one to conjecture an ELSV formula for these numbers. The arguments used in the present paper lift naturally to the spin setting. We present spin analogues for our main results below, without proof, since the arguments parallel those used earlier.

\begin{theorem}[ELSV formulas for one-part double spin Hurwitz numbers] \label{thm:spin}
Fix a positive integer $r$. For integers $g \geq 0$ and $n \geq 1$ with $(g,n) \neq (0,1)$ or $(0,2)$, the one-part double spin Hurwitz numbers satisfy the following formulas, where $d = \mu_1 + \cdots + \mu_n$.
\begin{itemize}
\item Chiodo classes on moduli spaces of spin curves
\begin{equation} \label{eq:spinELSVdChiodo}
h_{g; \mu_1, \ldots, \mu_n}^{(r), \textnormal{one-part}} = r^{2g-2+n} (dr)^{\frac{2g-1+n}{r} - (3g-3+n)} \int_{\overline{\mathcal M}_{g,n;-\mu}^{dr, d}} \frac{\textnormal{Chiodo}_{g,n}(dr, d; -\mu_1, \ldots, -\mu_n)}{\prod_{i=1}^n (1 - \mu_i \psi_i)}
\end{equation}

\item Tautological classes on moduli spaces of stable curves
\begin{equation} \label{eq:spinELSVdMgn}
h_{g; \mu_1, \ldots, \mu_n}^{(r), \textnormal{one-part}} = r^{2g-2+n} (dr)^{\frac{2g-1+n}{r} - (3g-3+n)} \int_{\overline{\mathcal M}_{g,n}} \frac{\epsilon_* \textnormal{Chiodo}_{g,n}(dr, d; -\mu_1, \ldots, -\mu_n)}{\prod_{i=1}^n (1 - \mu_i \psi_i)}
\end{equation}
\end{itemize}
\end{theorem}

\begin{remark}
\cref{thm:spin} generalises~\cref{thm:main}, but lacks ELSV formulas on the moduli space of stable maps to the classifying space $\mathcal{B} \mathbb{Z}_d$, as well as on the moduli space $\overline{\mathcal M}_{g,n+g}$. Regarding the former, the moduli space $\overline{\mathcal{M}}^{d,d}_{g,n; -\mu_1, \ldots, -\mu_n}$ appearing in~\cref{eq:ELSVdChiodo} bears a close relation with the moduli space $\overline{\mathcal{M}}_{g; -\mu_1, \ldots, -\mu_n}(\mathcal{B}\mathbb{Z}_d)$ appearing in~\cref{eq:ELSVdJPT}. On the other hand, the spin case that we consider in this section requires the more general class $\textnormal{Chiodo}_{g,n}(dr, d; -\mu_1, \ldots, -\mu_n)$, in which the first two parameters do not match unless $r = 1$. In that case, the relation to moduli spaces of stable maps to the classifying space $\mathcal{B} \mathbb{Z}_d$ is not expected. Regarding the latter, one should be able to compute a dilaton equation for the Chiodo classes involved in~\cref{eq:spinELSVdMgn} by combining the topological recursion techniques and results of~\cite{do-lei-nor} and~\cite{LPSZ}. This would then allow us to obtain an ELSV formula for one-part double spin Hurwitz numbers on the space $\overline{\mathcal M}_{g,n+g}$, analogous to~\cref{eq:ELSVdMgng} of~\cref{thm:main}. However, such a computation transcends the goal of this paper, so we do not perform it here.
\end{remark}

The exchange of ramification profiles used in~\cref{sec:comparison} can be invoked in the spin setting via the equation
\[
\frac{h_{g; p, \ldots, p}^{q\textnormal{-orbifold}, r\textnormal{-spin}}}{(d/p)!} = \frac{h_{g; q, \ldots, q}^{p\textnormal{-orbifold}, r\textnormal{-spin}}}{(d/q)!}.
\]
This leads directly to the following non-trivial relation between tautological intersection numbers on moduli spaces of curves, which generalises~\cref{thm:doubleorbifoldcomparison}.

\begin{theorem} \label{thm:spincomparison}
Let $p < q$ be positive integers and $d$ a multiple of them both.
For integers $g \geq 0$ and $d \geq 1$ with $(g,d/p), (g, d/q) \neq (0,1)$ or $(0,2)$, we have
\begin{multline*}
\frac{r^{d/p}}{(d/p)! \, p^{2g-2+\frac{d}{p}+\frac{d}{q}}} \int_{\overline{\mathcal M}_{g,d/p}} \frac{\epsilon_* \mathrm{Chiodo}(qr, q; -p, \ldots, -p)}{\prod_{i=1}^{d/p} (1 - \frac{p}{qr} \psi_i)} \\
= \frac{r^{d/q}}{(d/q)! \, q^{2g-2+\frac{d}{p}+\frac{d}{q}}} \left( \frac{(q/pr)^{\lfloor q/pr \rfloor}}{\lfloor q/pr \rfloor!} \right)^{d/q} \int_{\overline{\mathcal M}_{g,d/q}} \frac{\epsilon_* \mathrm{Chiodo}(pr, p; -q, \ldots, -q)}{\prod_{i=1}^{d/q} (1 - \frac{q}{pr} \psi_i)}.
\end{multline*}
\end{theorem}

\subsection{Evaluations of Chiodo integrals from  the spin case}
It is very natural to ask whether the type of polynomiality found by Goulden, Jackson, and Vakil in \cite{gou-jac-vak} for one-part Hurwitz numbers of \cref{eq:gen1} and \cref{eq:gen2} has a spin counterpart. In  particular, one could wonder whether that would give raise to statements involving generating series of Chiodo integrals (i.e. the spin counterparts of \cref{prop:genmu1} and \cref{prop:genmud}), and if so one would expect them to be again in terms of hyperbolic functions.

The spin counterpart of the polynomiality can easily be recovered from a semi-infinite wedge calculation (see e.g. \cite[Example 4.5]{SSZ}) and it reads:
\begin{equation}
h^{(r), \textnormal{one part}}_{g,\mu} = [z_1^r \dots z_b^r]. d^{b-1} \prod_{j=1}^b \SSS(d z_j) \prod_{i=1}^n \SSS\left(\mu_i z_{[b]})\right) \frac{z_{[b]}^{n-1}}{\SSS( z_{[b]})}, \quad \quad b= \frac{2g - 1 + n}{r},
\end{equation}
for $z_{[b]} = \sum_{j=1}^b z_j$, $d = \sum_{i=1} \mu_i$, and $\SSS(x) = \sinh(x/2)/(x/2)$ as before. Combining this with \cref{thm:spin} we immediately obtain the following.

\begin{proposition} \label{prop:final} Let $r$ be a positive integer. For integers $g \geq 0$ and a partition $\mu$ of length $n$ such that $2g - 2+ n >0$, let $d$ be the size of $\mu$. Then we have:
\begin{align*}
\int_{\overline{\mathcal M}_{g,n}}
 \!\!\!\! \frac{\epsilon_* \textnormal{Chiodo}^{[dr]}_{g,n}(dr, d; -\mu_1, \ldots, -\mu_n)}{\prod_{i=1}^n (1 - \mu_i \psi_i)}
=
  \frac{d^{3g - 4 + n}}{ r^{A(r)}} [z_1^r \dots z_b^r]. \prod_{j=1}^b \SSS(d z_j) \prod_{i=1}^n \SSS\left(\mu_i z_{[b]})\right) \frac{z_{[b]}^{n-1}}{\SSS( z_{[b]})}
\end{align*} 
for $A(r) = 1 - g + b$ and $b = (2g - 1+n)/r$.
\end{proposition}

One can indeed immediately specialise the proposition above to the cases $\mu= (1^d)$ and $\mu = (d)$, obtaining this way statements which for $r>1$ generalise \cref{prop:genmu1} and \cref{prop:genmud}, respectively. However, as the formulae do not structurally simplify, we prefer not to write them out explicitly. Instead, we would like to point out that \cref{prop:final} provides an explicit power-series expansion computation, easily software-implementable, to evaluate any Chiodo integral of the form
\begin{equation}
\int_{\overline{\mathcal M}_{g,n}}
 \!\!\!\!\!\!\! \frac{\epsilon_* \textnormal{Chiodo}^{[r]}_{g,n}(r, s; r-\mu_1, \ldots, r-\mu_n)}{\prod_{i=1}^n (1 - \mu_i \psi_i)}, \qquad \text{ for } \sum_{i=1} \mu_i = s \quad \text{ and } \quad s|r.
\end{equation}

\appendix

\section{Data} \label{sec:data}

We provide some calculations of one-part double Hurwitz numbers up to genus five below as polynomials in the parts, using the notation $\mu = (\mu_1, \ldots, \mu_n)$ and $d = \mu_1 + \cdots + \mu_n$. 
The original paper \cite[Corollary 3.3]{gou-jac-vak} also provides these polynomials up to genus five, expressing them instead in terms of the coefficients $S_{2j} = \sum \mu_i^{2j} - 1$.
These calculations have been however re-computed by means of cut and join equation as an independent check.  Note the structure 
\(
h_{g; \mu_1, \ldots, \mu_n}^{\textnormal{one-part}} = d^{2g-2+n} P_{g,n}(\mu_1^2, \ldots, \mu_n^2),
\)
for $P_{g,n}$ a symmetric polynomial of degree $2g$, mentioned in~\cref{sec:introduction}.
\begin{align*}
h_{0;\mu}^{\textnormal{one-part}} &= d^{n-2} \\
h_{1;\mu}^{\textnormal{one-part}} &= \frac{d^n}{24} (\mu_1^2 + \mu_2^2 + \cdots + \mu_n^2 - 1) \\
h_{2;\mu}^{\textnormal{one-part}} &= \frac{d^{n+2}}{5760} \Big( 3 \sum \mu_i^4 + 10 \sum \mu_i^2 \mu_j^2 - 10 \sum \mu_i^2 + 7 \Big) \\
h_{3;\mu}^{\textnormal{one-part}} &= \frac{d^{n+4}}{967680} \Big( 3 \sum \mu_i^6 + 21 \sum \mu_i^4\mu_j^2 + 70 \sum \mu_i^2\mu_j^2\mu_k^2 - 21 \sum \mu_i^4 - 70 \sum \mu_i^2\mu_j^2 + 49 \sum \mu_i^2 - 31 \Big) \\
h_{4;\mu}^{\textnormal{one-part}} &= \frac{d^{n+6}}{464486400} \Big( 5 \sum \mu_i^8 + 60 \sum \mu_i^6 \mu_j^2 + 126 \sum \mu_i^4 \mu_j^4 - 60 \sum \mu_i^6 - 420 \sum \mu_i^4 \mu_j^2 + 294 \sum \mu_i^4 \\
&\mkern-18mu+ 980 \sum \mu_i^2 \mu_j^2 - 620 \sum \mu_i^2 + 381 \Big) \\
h_{5;\mu}^{\textnormal{one-part}} &= \frac{d^{n+8}}{122624409600} \Big( 3 \sum \mu_i^{10} + 55 \sum \mu_i^8 \mu_j^2 + 198 \sum \mu_i^6 \mu_j^4 - 55 \sum \mu_i^8 - 660 \sum \mu_i^6 \mu_j^2 \\
&\mkern-18mu- 1386 \sum \mu_i^4 \mu_j^4 + 462 \sum \mu_i^6 + 3234 \sum \mu_i^4 \mu_j^2 - 2046 \sum \mu_i^4 - 6820 \sum \mu_i^2 \mu_j^2 + 4191 \sum \mu_i^2 - 2555 \Big)
\end{align*}

\bibliographystyle{plain}
\bibliography{GJV-conjecture.bib}

\end{document}